\documentclass[12pt]{article}

\usepackage{times}

\usepackage[a4paper,text={128mm,185mm}, centering]{geometry}

\usepackage{changepage}

\usepackage{titlesec}

\titleformat{\section}[hang]%
{\bfseries\large}{\thesection.}{1ex}{}%

\titleformat{\subsection}[hang]%
{\bfseries}{\thesubsection}{1ex}{}%

\usepackage{amsthm}
\usepackage{fancyhdr}
\usepackage[british]{babel}
\usepackage[utf8]{inputenc}
\usepackage[T1]{fontenc}
\usepackage[notext,nomathscript,fulloldstylenums]{kpfonts}
\usepackage[cal=boondox,bb=boondox]{mathalfa}
\usepackage[mathscr]{eucal}
\usepackage{amsmath}
\usepackage{amssymb}
\usepackage{enumitem}
\usepackage{etoolbox}
\usepackage{tikz-cd}
\usetikzlibrary{decorations.pathmorphing}
\usepackage{dsfont}
\usepackage{calc}
\usepackage{mathtools}
\usepackage{xstring}
\usepackage{pict2e,picture} 
\usepackage{hyperref}
\usepackage[nameinlink]{cleveref}

\theoremstyle{definition}
\newtheorem{definition}[subsection]{\definitionautorefname}
\newcommand{\definitionautorefname}{Definition}

\newtheorem{construct}[subsection]{\constructautorefname}
\newcommand{\constructautorefname}{Construction}

\newtheorem{conj}[subsection]{\conjautorefname}
\newcommand{\conjautorefname}{Conjecture}

\theoremstyle{remark}
\newtheorem{remark}[subsection]{\remarkautorefname}
\newcommand{\remarkautorefname}{Remark}
\newtheorem{subremark}[subsubsection]{\remarkautorefname}
\newtheorem{exmp}[subsection]{\exmpautorefname}
\newcommand{\exmpautorefname}{Example}
\newtheorem{subexmp}[subsubsection]{\exmpautorefname}

\newcommand{\warningautorefname}{Warning}
\newtheorem{subwarning}[subsubsection]{\warningautorefname}
\newtheorem{notation}[subsection]{\notationautorefname}
\newcommand{\notationautorefname}{Notation}

\theoremstyle{plain} \newtheorem{thm}[subsection]{\thmautorefname}
\newcommand{\thmautorefname}{Theorem}

\newtheorem{supthm}{\thmautorefname}
\newtheorem{corlr}[subsection]{\corlrautorefname}
\newcommand{\corlrautorefname}{Corollary}
\newtheorem{subcorlr}[subsubsection]{\corlrautorefname}

\newcommand{\prorautorefname}{Property}

\newtheorem{pros}[subsection]{\prosautorefname}
\newcommand{\prosautorefname}{Proposition}

\newtheorem{lemma}[subsection]{\lemmaautorefname}
\newcommand{\lemmaautorefname}{Lemma}

\newcommand{\scholiumautorefname}{Scholium}

\Crefname{definition}{\definitionautorefname}{\definitionautorefname{}s}
\Crefname{construct}{\constructautorefname}{\constructautorefname{}s}
\Crefname{conj}{\conjautorefname}{\conjautorefname{}s}
\Crefname{remark}{\remarkautorefname}{\remarkautorefname{}s}
\Crefname{exmp}{\exmpautorefname}{\exmpautorefname{}s}
\Crefname{warning}{\warningautorefname}{\warningautorefname{}s}
\Crefname{notation}{\notationautorefname}{\notationautorefname{}s}
\Crefname{thm}{\thmautorefname}{\thmautorefname{}s}
\Crefname{corlr}{\corlrautorefname}{\corlrautorefname{}s}
\Crefname{pror}{\prorautorefname}{\prorautorefname{}s}
\Crefname{pros}{\prosautorefname}{\prosautorefname{}s}
\Crefname{lemma}{\lemmaautorefname}{\lemmaautorefname{}s}
\Crefname{scholium}{\scholiumautorefname}{\scholiumautorefname{}s}

\makeatletter
\DeclareRobustCommand\widecheck[1]{{\mathpalette\@widecheck{#1}}}
\def\@widecheck#1#2{%
  \setbox\z@\hbox{\m@th$#1#2$}%
  \setbox\tw@\hbox{\m@th$#1%
    \widehat{%
      \vrule\@width\z@\@height\ht\z@
      \vrule\@height\z@\@width\wd\z@}$}%
  \dp\tw@-\ht\z@
  \@tempdima\ht\z@ \advance\@tempdima2\ht\tw@ \divide\@tempdima\thr@@
  \setbox\tw@\hbox{%
    \raise\@tempdima\hbox{\scalebox{1}[-1]{\lower\@tempdima\box
        \tw@}}}%
  {\ooalign{\box\tw@ \cr \box\z@}}}
\makeatother

\makeatletter
\DeclareRobustCommand{\bbDelta}{{\mathpalette\bb@Delta\relax}}
\newcommand{\bb@Delta}[2]{%
  \begingroup
  \sbox\z@{$\m@th#1\Delta$}%
  \dimendef\Dht=6 \dimendef\Dwd=8
  \setlength{\Dwd}{\wd\z@}%
  \setlength{\Dht}{\ht\z@}%
  \begin{picture}(\Dwd,\Dht)
  \put(0,0){$\m@th#1\Delta$}
  \put(.42\Dwd,.7\Dht){\line(10,-26){.25\Dht}}
  \end{picture}%
  \endgroup
}
\makeatother
\makeatletter
\DeclareRobustCommand{\bbGamma}{{\mathpalette\bb@Gamma\relax}}
\newcommand{\bb@Gamma}[2]{%
  \begingroup
  \sbox\z@{$\m@th#1\Gamma$}%
  \dimendef\Dht=6 \dimendef\Dwd=8
  \setlength{\Dwd}{\wd\z@}%
  \setlength{\Dht}{\ht\z@}%
  \begin{picture}(\Dwd,\Dht)
  \put(0,0){$\m@th#1\Gamma$}
  \put(.47\Dwd,.025\Dht){\line(0,1){.9\Dht}}
  \end{picture}%
  \endgroup
}
\makeatother
\makeatletter
\DeclareRobustCommand{\bbOmega}{{\mathpalette\bb@Omega\relax}}
\newcommand{\bb@Omega}[2]{%
  \begingroup
  \sbox\z@{$\m@th#1\Omega$}%
  \dimendef\Dht=6 \dimendef\Dwd=8
  \setlength{\Dwd}{\wd\z@}%
  \setlength{\Dht}{\ht\z@}%
  \begin{picture}(\Dwd,\Dht)
  \put(0,0){$\m@th#1\Omega$}
  \put(.275\Dwd,.125\Dht){\line(0,1){.775\Dht}}
  \put(.6825\Dwd,.17\Dht){\line(0,1){.725\Dht}}
  \end{picture}%
  \endgroup
}
\makeatother
\makeatletter
\DeclareRobustCommand{\bbUpsilon}{{\mathpalette\bb@Upsilon\relax}}
\newcommand{\bb@Upsilon}[2]{%
  \begingroup
  \sbox\z@{$\m@th#1\Omega$}%
  \dimendef\Dht=6 \dimendef\Dwd=8
  \setlength{\Dwd}{\wd\z@}%
  \setlength{\Dht}{\ht\z@}%
  \begin{picture}(\Dwd,\Dht)
  \put(0,0){$\m@th#1\Upsilon$}
  \put(.55\Dwd,.033\Dht){\line(0,1){.75\Dht}}
  \end{picture}%
  \endgroup
}
\makeatother

\newcommand\mapstofonc{\mathrel{\ooalign{$\rightsquigarrow$\cr%
  \kern-.105ex\raise.325ex\hbox{\scalebox{1}[0.388]{$\mid$}}\cr}}}

\makeatletter
\newcommand{\addchar}[2]{%
  \@tfor\letter:=#1\do{%
    \letter#2
  }%
}
\makeatother

\DeclareMathOperator{\id}{id}%
\newcommand{\cat}[1]{\mathfrak{#1}}
\newcommand{\op}[1]{{#1}^{\mathrm{op}}}
\newcommand{\func}[1]{\mathcal{\addchar{#1}{\!}}\,}
\newcommand{\internalcat}[1]{\mathbb{#1}}

\newcommand{\lmtimes}{\mathop{\times}\limits}

\newcommand{\grothco}[1][{}]{\int^{#1}}%
\newcommand{\infgrpds}{\cat{\infty\textnormal{-}Grpd}}
\newcommand{\inflcats}[1]{\cat{\mathnormal{(\infty,#1)}\textnormal{-}Cat}}
\newcommand{\infcats}{\inflcats{1}}

\newcommand{\funcs}[2]{\bigl\{\cat{#1},\cat{#2}\bigr\}}

\newcommand{\seg}[1]{\cat{Seg}_{#1}}

\newcommand{\inrt}{\text{\textnormal{inrt}}}
\newcommand{\activ}{\text{\textnormal{act}}}
\newcommand{\elem}{\text{\textnormal{el}}}

\DeclareMathOperator{\ev}{ev}

\DeclareMathOperator{\const}{\func{const}}
\DeclareMathOperator*{\colim}{colim}

\DeclarePairedDelimiter{\uly}{\lvert}{\rvert}
\DeclarePairedDelimiter{\restr}{{}}{\rvert}

\DeclarePairedDelimiter{\characmap}{\ulcorner}{\urcorner}

\title{\vskip 5pt \bf ALL SEGAL OBJECTS ARE GENERALISED MONADS IN
  SPANS}
\author{\itshape\bfseries { David KERN}}
\date{}

\begin{document}
\maketitle

\cfoot{}
\thispagestyle{empty}
\vskip 25pt
\begin{adjustwidth}{0.5cm}{0.5cm} {\small {\bf R\'esum\'e.} Nous
    étendons la construction par Barwick et Haugseng d'une
    $\infty$-catégorie double de correspondances dans une
    $\infty$-catégorie $\cat{C}$ admettant les produits fibrés à des
    formes plus générales : pour une large classe de patrons
    algébriques $\cat{P}$, nous définissons une $\infty$-catégorie
    $\cat{P}$-monoïdale de correspondances $\cat{P}$-modelées dans
    $\cat{C}$, et identifions les $\cat{P}$-monades dedans avec les
    $\cat{P}$-objets de Segal dans $\cat{C}$. Pour le patron
    cellulaire $\op{\Theta}$, cela recouvre une reformulation
    homotopique de la définition originale de Batanin des
    $\omega$-catégories faibles, et en général peut être vu comme une
    variante des multicatégories généralisées de Burroni, Hermida,
    Leinster et Cruttwell--Shulman.\\
    {\bf Abstract.} We extend Barwick's and Haugseng's construction of
    the double $\infty$-category of spans in a pullback-complete
    $\infty$-category $\cat{C}$ to more general shapes: for a large
    class of algebraic patterns $\cat{P}$, we define a
    $\cat{P}$-monoidal $\infty$-category of $\cat{P}$-shaped spans in
    $\cat{C}$, and we identify $\cat{P}$-monads in it with Segal
    $\cat{P}$-objects in $\cat{C}$. For the cell pattern
    $\op{\Theta}$, this recovers a homotopical reformulation of
    Batanin's original definition of weak $\omega$-categories, and in
    general can be seen as a variant of the generalised
    multicategories
    of Burroni, Hermida, Leinster and Cruttwell--Shulman.\\
    {\bf Keywords.} Segal objects, spans, double categories, weak
    $\omega$-categories, multicategories.\\
    {\bf Mathematics Subject Classification (2020).} 18N65, 18N70.  }
\end{adjustwidth}


\section{Introduction}
\label{sec:introduction}

\renewcommand{\thesupthm}{\Alph{supthm}}

\subsection{Algebraic structures for higher categories}
\label{sec:algebr-struct-high}

The various definitions of higher categories come in two families:
algebraic definitions specify the minimal amount of shape data (for
$\ell$-categories, an $\ell$-graph, comprised only of elementary
cells) and add the structure of all the composition operations and
their higher coherences, while geometric definitions start from a
bigger shape containing all the possible pasting diagrams of cells and
simply impose conditions to ensure that they come from decompositions
into compatible elementary cells.

For example, the standard definition of an internal category, in a
category $\cat{C}$ admitting finite pullbacks, is as a
$\op{\bbDelta}$-shaped object $X_{\bullet}$ of $\cat{C}$ --- where
$\bbDelta$ is the category of free categories on $1$-dimensional
pasting diagrams, that is sequences of composable arrows ---
satisfying the Segal decomposition condition which expresses each
value $X_{n}$ on a pasting of $n$ consecutive arrows as
$X_{1}\times_{X_{0}}\dots\times_{X_{0}}X_{1}$. This can be
reinterpreted in a more algebraic way as giving a graph
$\restr{X_{\bullet}}_{\{[0],[1]\}}$ in $\cat{C}$ and a certain kind of
algebra structure on it, subject to the simplicial identities. To make
good sense of this algebra structure, it was noticed
by~\cite{bénabou67:_introd} that a graph in $\cat{C}$ is nothing but
an endomorphism in the bicategory (or better, the double category) of
spans in $\cat{C}$, and the required algebra structure is none other
than a structure of monad on this endomorphism.

For \emph{strict} higher categories, the situation generalises
directly: on the one hand, \cite{joyal97:_disks_theta} introduced a
category $\Theta$ of free $\omega$-categories on $\omega$-categorical
pasting diagrams, so that strict $\omega$-categories in any category
with fibre products $\cat{C}$ are exactly $\cat{C}$-valued presheaves
on $\Theta$ satisfying a Segal condition. On the other hand,
\cite{batanin98:_monoid_globul_categ_as_natur} constructed an internal
(strict) $\omega$-category in $\cat{Cat}$ (a globular object in
$\cat{Cat}$ equipped with compositions)
$\internalcat{Span}_{\infty}(\cat{C})$ of infinitely iterated spans in
$\cat{C}$, so that globular monads in it are exactly strict
$\omega$-categories internal to $ \cat{C}$.

The key insight of~\cite{batanin98:_monoid_globul_categ_as_natur} is
then that, using the higher structure naturally present in globular
categories, one can refine the teminal globular operad to a suitable
contractible globular operad $\mathcal{A}_{\infty}^{\mathbb{G}}$,
which contains enough coherence data for
$\mathcal{A}_{\infty}^{\mathbb{G}}$-algebras in
$\internalcat{Span}_{\infty}(\cat{C})$ to be a good definition of weak
$\omega$-categories in $\cat{C}$.

While the presence of higher cells in globular sets allows one to make
sense of $\mathcal{A}_{\infty}^{\mathbb{G}}$ as an algebraic
resolution of the terminal globular operad, eschewing any homotopical
machinery, formulating things in a setting of homotopy theory allows
many constructions to become simpler, and more widely
applicable. Indeed, the logic of using the higher cells to tame the
infinite towers of coherences needed for a resolution only works for
full $\omega$-categories, but breaks down if trying to define weak
$\ell$-categories for some $\ell<\omega$. Nonetheless,
\cite{haugseng21:_segal} showed that the situation for (weak)
$1$-categories can be dealt with using $\infty$-categories: category
objects in an $(\infty,1)$-category $\cat{C}$ are identified with
homotopy-coherent monads in the double $(\infty,1)$-category of spans
in $\cat{C}$.

In this note, we extend this result (as a direct application
of~\Cref{thm:monads-spans-segal-objs}) to a characterisation of
$\ell$-category objects as $\ell$-globular monads in $\ell$-times
iterated spans, which both extends Batanin's definition of weak
$\omega$-categories to one for weak $\ell$-categories for any
$\ell\leq\omega$, and also simplifies it by removing the need to
resolve the terminal globular $\infty$-operad by a more complicated
one.

\subsection{Multicategories and algebraic patterns}
\label{sec:mult-algebr-patt}

In order to understand how to construct categories of generalised
spans, let us switch gears to another categorical structure that can
be defined in a similar way: multicategories, or coloured operads. It
was noticed
by~\cite{burroni71:_t_catég,hermida04:_fibrat,leinster98:_gener_operad_multic}
that multicategories can be defined as monads in a double category of
Kleisli $\func{M}$-spans, where $\func{M}$ is the ``free monoid''
monad on $\cat{Set}$, fitting in a more general framework of
$\func{T}$-multicategories, for a cartesian monad $\func{T}$, as
monads in a double category of Kleisli $\func{T}$-spans, whose
morphisms are the spans twisted by $\func{T}$ on their source, and
whose composition uses $\func{T}$'s monad structure. In particular,
Batanin's globular operads can also be obtained in this way.

Unfortunately, this double category of Keisli $\func{T}$-spans is not
characterised by a clear universal property (see~\cite[Remark
4.2]{cruttwell10}), which makes constructing it in the
$\infty$-categorical world very difficult. Because of this, we will
instead use a different kind of structure to organise the generalised
spans.

To explain the idea, let us keep focusing of the example of
multicategories. An $\func{M}$-span from a set $Y_{0}$ to a set
$X_{0}$ is given by a span $\func{M}Y_{0}\leftarrow X_{1}\to X_{0}$,
which we interpret as a multispan (as championed by~\cite{baas19} for
the study of hyperstructures) of some arbitrary arity $a+1$, one of
whose legs (the root) goes to $X_{0}$ and the $a$ others (the leaves)
to $Y_{0}$. To compose it with an $\func{M}$-span
$\func{M}Z_{0}\leftarrow Y_{1}\to Y_{0}$, one forms the
$\func{M}$-span
\begin{equation}
  \label{eq:compose-multispans}
  \begin{tikzcd}[cramped]
    & & & \func{M}Y_{1}\lmtimes_{\func{M}Y_{0}}X_{1} \arrow[dr]
    \arrow[dl] & & \\
    & & \func{M}Y_{1} \arrow[dr] \arrow[dl] & & X_{1} \arrow[dr]
    \arrow[dl] & \\
    \func{M}Z_{0} & \func{M}^{2}Z_{0} \arrow[l,"\mu"'] & &
    \func{M}Y_{0} & & X_{0}
  \end{tikzcd}
\end{equation}
expressing that one takes $a$ copies of the multispan corresponding to
$\func{M}Z_{0}\leftarrow Y_{1}\to Y_{0}$ and glues their distinguished
roots to the various leaves of
$\func{M}Y_{0}\leftarrow X_{1}\to X_{0}$.

As is usual in operad theory, one also, instead of blowing up the
situation globally, glue a single new span to one leaf of
$\func{M}Y_{0}\leftarrow X_{1}\to X_{0}$; the composition operation
defined in this way, leaf by leaf, will no longer be a categorical
composition, but indeed an operadic one. Thus, multispans can be
organised, instead of in a double category, in a categorical operad
(internal category in the category of operads).

While there are many different approaches to operadic structures in
the $1$-categorical setting, in the $\infty$-categorical one a very
convenient and powerful framework is that of the algebraic patterns
of~\cite{chu21:_homot_segal}, which extract the necessary data on a
category of shapes to speak of Segal decompositions (inert morphisms
from elementary objects) and keep additional algebraic operations
(active morphisms): in other words, they give a geometric presentation
of $\infty$-operadic structure, while remembering what is the
algebraic part. In the approach that we will follow in this note, the
choice of an algebraic pattern will play the role of the choice of the
cartesian monad $\func{T}$ in the story sketched above.

We will then construct in~\cref{sec:gener-mult}, for any algebraic
pattern $\cat{P}$ (satisfying the very mild condition of soundness ---
that will be verified in all examples we know of, in particular
in~\cref{sec:globl-satur-cell} for $\omega$-categories) and any
complete enough $(\infty,1)$-category $\cat{C}$, a Segal
$\cat{P}$-object $\internalcat{Span}_{\cat{P}}(\cat{C})$ in $\infcats$
of $\cat{P}$-shaped spans in $\cat{C}$, by adapting the construction
of~\cite{haugseng18:_iterat} with the ideas raised in~\cite{street00}
and expanded upon in~\cite[Example 4.8]{weber07:_yoned_struc}. We will
continue in~\cref{sec:monads-spans} by showing the result promised
above:
\begin{supthm}(\emph{cf.}~\Cref{thm:monads-spans-segal-objs})
  \label{thm:main-result}
  There is an equivalence between $\cat{P}$-monads in the
  $\cat{P}$-monoidal $\infty$-category
  $\internalcat{Span}_{\cat{P}}(\cat{C})$ and Segal $\cat{P}$-objects
  in $\cat{C}$.
\end{supthm}

Then, in~\cref{sec:some-exampl-flav}, we will spell out the meaning of
this result in examples of interest, in particular
$(\infty,\ell)$-categories.

\subsection{Aknowledgements}
\label{sec:aknowledgements}

This note was closely inspired by the ideas
of~\cite{haugseng18:_iterat,haugseng21:_segal} and~\cite{street00},
and would not exist without the insights developed in these
works. Thanks are also due to Damien Calaque for discussions about
algebras in iterated spans, to Hugo Pourcelot for conversations about
differently-shaped spans, and to Reuben Stern for useful comments
about the interpretation of fibrous $\op{\Theta}$-patterns. Many
thanks to Jan Steinebrunner and Thomas Blom for pointing out the
necessity of soundness in~\cref{lemma:spans-are-cocart-fib} and the
automaticity of global saturation, and to Jan especially sharing
material on sound patterns and a proof of global saturation. I thank
the anonymous reviewer for several corrections and improvements.

The author was supported by the Göran Gustafsson Foundation for
Research in Natural Sciences and Medicine.

\section{Algebraic and fibrous patterns}
\label{sec:weak-segal-fibr}

\begin{definition}[Algebraic pattern]
  An \textbf{algebraic pattern} is a diagram of inclusions of
  $(\infty,1)$-categories
  \begin{equation}
    \label{eq:algpatrn-def}
    \begin{tikzcd}
      & & \cat{P} & \\
      \cat{P}^{\elem} \arrow[r,hook] & \cat{P}^{\inrt} \arrow[ur,hook]
      & & \cat{P}^{\activ} \arrow[ul,hook']
    \end{tikzcd}
  \end{equation}
  where the wide sub-$(\infty,1)$-categories
  $(\cat{P}^{\inrt},\cat{P}^{\activ})$ form an orthogonal
  factorisation system on $\cat{P}$ and
  $\cat{P}^{\elem}\subset\cat{P}^{\inrt}$ is a full
  sub-$(\infty,1)$-category.
\end{definition}

The \textbf{inert} arrows (those in $\cat{P}^{\inrt}$) are denoted as
$\rightarrowtail$ and the \textbf{active} ones (those in
$\cat{P}^{\activ}$) are denoted as $\rightsquigarrow$, while the
objects in $\cat{P}^{\elem}$ are known as \textbf{elementary}.

\begin{notation}
  For any $P\in\cat{P}$, we write
  $\cat{P}^{\elem}_{P/}=\cat{P}^{\elem}\times_{\cat{P}}\cat{P}^{\elem}_{P/}$.
  
  An $(\infty,1)$-category $\cat{C}$ is said to be
  \textbf{$\cat{P}$-complete} if it admits limits of diagrams of shape
  $\cat{P}^{\elem}_{P/}$ for any $P\in\cat{P}$.
\end{notation}

\begin{definition}[Segal object]
  Let $\cat{P}$ be an algebraic pattern and $\cat{C}$ a
  $\cat{P}$-complete $(\infty,1)$-category. A \textbf{Segal
    $\cat{P}$-object} in $\cat{C}$ is a functor
  $\func{X}\colon\cat{P}\to\cat{C}$ such that
  $\restr{\func{X}}_{\cat{P}^{\inrt}}$ is the right Kan extension of
  its restriction to $\cat{P}^{\elem}$, which means that for any
  $P\in\cat{P}$, the canonical arrow
  \begin{equation}
    \label{eq:segal-cond-def}
    \func{X}(P)\to\lim_{E\in\cat{P}^{\elem}_{P/}}\func{X}(E)
  \end{equation}
  is an equivalence.
\end{definition}

The full sub-$(\infty,1)$-category of the functor
$(\infty,1)$-category $\funcs{P}{C}$ on the Segal objects is denoted
$\seg{\cat{P}}(\cat{C})$.

\begin{exmp}[Product patterns]
  \label{exmp:product-pattrn}
  The $(\infty,1)$-category of algebraic patterns admits all limits,
  which can be computed at the level of the underlying
  $(\infty,1)$-categories. In particular, it admits products, and
  these are compatible with currying, in that if $\cat{P}$ and
  $\cat{Q}$ are two algebraic patterns and $\cat{C}$ is
  $\cat{P}\times\cat{Q}$-complete, then $\seg{\cat{Q}}(\cat{C})$ is
  $\cat{P}$-complete and there is an equivalence
  $\seg{\cat{P}\times\cat{Q}}(\cat{C})
  \simeq\seg{\cat{P}}(\seg{\cat{Q}}(\cat{C}))$.
\end{exmp}

\begin{exmp}[$\cat{P}$-graphs]
  \label{exmp:ulying-graphs}
  As observed in~\cite[beginning of \S{}8]{chu21:_homot_segal}, any
  algebraic pattern $\cat{P}$ restricts to a pattern structure on
  $\cat{P}^{\inrt}$, whose only active morphisms are the equivalences,
  and further restricts to $\cat{P}^{\elem}$. Evidently, the
  restriction--right Kan extension adjunctions along
  $\cat{P}^{\inrt,\elem}=\cat{P}^{\elem}\hookrightarrow\cat{P}^{\inrt}$
  and
  $\cat{P}^{\elem,\elem}=\cat{P}^{\elem}\hookrightarrow\cat{P}^{\elem}$
  induce equivalences
  $\seg{\cat{P}^{\inrt}}(\cat{C})\simeq\funcs{P^{\elem}}{C}$ and
  $\seg{\cat{P}^{\elem}}(\cat{C})\simeq\funcs{P^{\elem}}{C}$ for any
  $\cat{P}$-complete $(\infty,1)$-category $\cat{C}$. We will refer to
  (necessarily Segal) $\cat{P}^{\elem}$-objects as
  \textbf{$\cat{P}$-graph}, and to the restriction of a Segal
  $\cat{P}$-object to $\cat{P}^{\elem}$ as its \textbf{underlying
    $\cat{P}$-graph}.
\end{exmp}

When $\cat{C}$ is $\infcats$, Segal $\cat{P}$-objects
$\cat{P}\to\infcats$ can also be seen as categoy objects in the
$\infty$-category $\seg{\cat{P}}(\infgrpds)$ of Segal
$\cat{P}$-$\infty$-groupoids, and as such will generally be written as
$\internalcat{X},\internalcat{Y},\dots$, in the font reserved for
internal categories. Such an object
$\internalcat{X}\colon\cat{P}\to\infcats$ can be recast as a
cocartesian fibration $\cat{X}=\grothco[P]{\internalcat{X}}\to\cat{P}$
satisfying the Segal condition for its fibres. We call such fibrations
\textbf{Segal $\cat{P}$-fibrations}. A certain weakening of this
notion turns out to be extremely useful, in particular to define lax
morphisms between Segal fibrations.

\begin{definition}[Fibrous pattern]
  Let $\cat{P}$ be an algebraic pattern. A \textbf{fibrous
    $\cat{P}$-pattern} is an $(\infty,1)$-functor
  $\func{f}\colon\cat{X}\to\cat{P}$ such that:
  \begin{enumerate}
  \item for every object $X\in\cat{X}$, every inert arrow
    $i\colon\func{f}X\to P$ in $\cat{P}$ admits a
    $\func{f}$-cocartesian lift $i_{!}\colon X\to i_{!}X$;
  \item for every $P\in\cat{P}$, the commutative square
    \begin{equation}
      \label{eq:fibrous-pattrn}
      \begin{tikzcd}[column sep=large]
        \cat{X}\lmtimes_{\cat{P}}\cat{P}^{\activ}_{/P} \arrow[r]
        \arrow[d] & \lim_{E\in\cat{P}^{\mathrm{el}}_{P/}}
        \cat{X}\lmtimes_{\cat{P}}\cat{P}^{\activ}_{/E} \arrow[d] \\
        \cat{P}^{\activ}_{/P}
        \arrow[r,"\lim_{E\in\cat{P}^{\mathrm{el}}_{P/}}(P\rightarrowtail
        E)_{!}"'] &
        \lim_{E\in\cat{P}^{\mathrm{el}}_{P/}}\cat{P}^{\activ}_{/E}
      \end{tikzcd}
    \end{equation}
    is cartesian.
  \end{enumerate}
  A \textbf{morphism of fibrous $\cat{P}$-patterns} from
  $\cat{X}\to\cat{P}$ to $\cat{Y}\to\cat{P}$ is an $\infty$-functor
  $\cat{X}\to\cat{Y}$ over $\cat{P}$ preserving cocartesian arrows
  over inert arrows of $\cat{P}$.
\end{definition}

Morphisms from $\cat{X}\to\cat{P}$ to $\cat{Y}\to\cat{P}$ are also
called $\cat{X}$-algebras in $\cat{Y}$, and their
$(\infty,1)$-category is denoted $\cat{Alg}_{\cat{X}}(\cat{Y})$.

\begin{lemma}[{\cite[Lemma 9.10]{chu21:_homot_segal}}]
  The domain of fibrous pattern $\func{f}\colon\cat{X}\to\cat{P}$
  admits a structure of algebraic pattern, where an arrow is active if
  it is over an active arrow of $\cat{P}$, inert if it is
  $\func{f}$-cocartesian and lies over an inert arrow, and an object
  is elementary if it lies over an elementary of $\cat{P}$. In
  particular, Segal morphisms between (sources of) fibrous
  $\cat{P}$-patterns are exactly their morphisms of fibrous
  $\cat{P}$-patterns.
\end{lemma}

If $\cat{X}\to\cat{P}$ and $\cat{Y}\to\cat{P}$ are Segal
$\cat{P}$-fibrations, with corresponding $\cat{P}$-monoidal
$(\infty,1)$-categories $\internalcat{X}$ and $\internalcat{Y}$,
morphisms of fibrous patterns $\cat{X}\to\cat{Y}$ can be seen as the
lax morphisms $\internalcat{X}\to\internalcat{Y}$.

\begin{definition}[$\cat{P}$-Monads]
  \label{def:monad}
  Let $\func{f}\colon\cat{X}\to\cat{P}$ be a fibrous
  $\cat{P}$-pattern. A \textbf{$\cat{P}$-monad} in $\cat{X}$ is a
  morphism from the terminal (weak) Segal $\cat{P}$-fibration
  $\cat{P}\xrightarrow{\id}\cat{P}$ to $\func{f}$.
\end{definition}

\begin{remark}
  When $\cat{P}$ is the algebraic pattern ${\op{\bbDelta}}^{\natural}$
  for internal categories (recalled
  in~\cref{sec:multiple-cats-iter-spans}), this recovers the usual
  definition of monads in double $\infty$-categories. More generally,
  for enrichable patterns, typically those denoted with a
  $(-)^{\natural}$ superscript in~\cite{chu21:_homot_segal},
  $\cat{P}$-monads can be thought of as a kind of $\cat{P}$-shaped
  generalisation of monads, as will be explained in examples
  in~\cref{sec:some-exampl-flav}. On the other hand, for the
  associated cartesian patterns $\cat{P}^{\flat}$, then
  $\cat{P}^{\flat}$-monads correspond rather to a kind of monoids,
  recovering the etymologically motivating observation of~\cite[\S{}
  (5.4.1)]{bénabou67:_introd}.
\end{remark}

An important technical condition on algebraic patterns will be that of
soundness from~\cite{barkan22:_envel_algeb_patter}, which we will
introduce with an alternative (equivalent) presentation due
to~\cite{barkan25:_sound_algeb_patter} that is convenient to handle.

\begin{construct}
  Let $\func{f}\colon\cat{O}\to\cat{P}$ be a morphism of algebraic
  patterns. The \textbf{inert factorisation $(\infty,1)$-category} of
  $\func{f}$ at $O\in\cat{O}$ and
  $\upsilon\colon\func{f}O\rightarrowtail
  E\in\cat{P}^{\elem}_{\func{f}O/}$ is
  \begin{equation}
    \label{eq:inert-facto-def}
    \cat{Fact}^{\inrt}_{\func{f}}(O,\upsilon)
    \coloneqq\cat{O}^{\elem}_{O/}\lmtimes_{\cat{P}^{\inrt}_{\func{f}O/}}
    \cat{Fact}_{\cat{P}^{\inrt}}(\upsilon)
  \end{equation}
  where
  $\cat{Fact}_{\cat{P}^{\inrt}}(\upsilon)
  =\{\upsilon\}\times_{\cat{Ar}_{\inrt}(\cat{P})}
  \funcs{\mathbb{3}}{\cat{P}^{\inrt}}$ is the $(\infty,1)$-category of
  factorisations of $\upsilon$ in $\cat{P}^{\inrt}$ (the pullback
  being defined relative to the functor $\mathbb{2}\to\mathbb{3}$
  which is the unique endpoints-preserving one, encoding composition
  of a composable pair).
\end{construct}

\begin{definition}[Sound patterns]
  An algebraic pattern $\cat{P}$ is \textbf{sound} if for every
  $f\colon P\rightsquigarrow P^{\prime}$ in
  $\cat{Ar}_{\activ}(\cat{P})$ and any
  $h\colon\ev_{0}f=P\rightarrowtail E$ in $\cat{P}^{\elem}_{P/}$ the
  $\infty$-category
  $\cat{Fact}^{\inrt}_{\ev_{0}\colon\cat{Ar}_{\activ}(\cat{P})\to\cat{P}}
  (f\colon P\rightsquigarrow P^{\prime},h\colon P\rightarrowtail E)$
  is contractible.
\end{definition}

One can see upon examination that the $(\infty,1)$-category
$\cat{Fact}^{\inrt}_{\ev_{0}}(f,h)$, whose objects are diagram of the form
\begin{equation}
  \label{eq:inert-facto-objs}
  \begin{tikzcd}
    P \arrow[d,rightsquigarrow,"f"'] \arrow[r,tail] & M
    \arrow[d,rightsquigarrow] \arrow[r,tail] & E\text{,} \\
    P^{\prime} \arrow[r,tail] & E^{\prime} &
  \end{tikzcd}
\end{equation}
is equivalent to that denoted
$\cat{P}^{\elem}_{h}(f)
=\cat{P}^{\elem}_{P^{\prime}/}\times_{\cat{P}^{\inrt}_{P/}}
(\cat{P}^{\inrt}_{P/})_{/h}$ of~\cite[Lemma
3.3.9. (2)]{barkan22:_envel_algeb_patter}, so that this definition
recovers the notion of soundness from Definition
3.3.4 of \emph{ibid}.

\begin{remark}
  If $\cat{P}$ is a sound algebraic pattern, fibrous
  $\cat{P}$-patterns coincide with the more familiar (as a
  generalisation of the $\infty$-operads
  of~\cite{lurie17:_higher_algeb}) weak Segal $\cat{P}$-fibrations
  (also called $\cat{P}$-operads) of~\cite{chu21:_homot_segal}. In
  particular, fibrous $\cat{P}$-patterns which are also cocartesian
  fibrations are then the same thing as Segal $\cat{P}$-fibrations.
\end{remark}

\begin{lemma}
  \label{lemma:colim-sound-pttrns}
  A filtered colimit of sound algebraic patterns is sound.
\end{lemma}

\begin{proof}
  Let $\cat{I}$ be a filtered $(\infty,1)$-category and
  $\func{P}\colon\cat{I}\to\cat{AlgPatt}$ be a diagram of algebraic
  patterns all of which are sound. For any pattern $\cat{P}$, the
  every $\infty$-category
  $\cat{Fact}^{\inrt}_{\ev_{0}\colon\cat{Ar}_{\activ}(\cat{P})\to\cat{P}}(f,h)$
  are pullbacks of powers of $\cat{P}$ by finite categories, so finite
  weighted limits, and hence their construction commutes with filtered
  colimits (in $\infcats$, and since limits and filtered colimits of
  algebraic patterns are computed in $\infcats$). Since a filtered
  colimit of contractible $(\infty,1)$-categories is
  contractible
  , as all the terms $\func{P}I$ are sound, we do obtain that the
  inert factorisation $\infty$-categories of
  $\colim_{I\in\cat{I}}\func{P}I$ are contractible.
\end{proof}

Finally, we describe a property of algebraic patterns which will be
paramount for the construction and Segality of the categories of
spans.

\begin{notation}[Co-internalisation of a category]
  \label{notation:co-internalisation}
  For any $(\infty,1)$-category $\cat{E}$, we will let $\cat{E}_{-/}$
  denote the $\infty$-functor $\op{\cat{E}}\to\infcats$ taking an
  object $E\in\cat{E}$ to the slice $\cat{E}_{E/}$ and an arrow
  $f\colon E\to E^{\prime}$ to the \textbf{codependent coproduct} (or,
  plainly, precomposition by $f$)
  $\Sigma^{f}=(-\circ f)\colon\cat{E}_{E^{\prime}/}\to\cat{E}_{E/}$. We
  refer to it as the \textbf{co-internalisation} of $\cat{E}$, though
  it differs from the internalisation of $\op{\cat{E}}$ considered
  in~\cite{street00} in that the latter, defined for $\cat{E}$
  admitting pushouts, has functoriality along $f$ given by the
  right-adjoint (co-base change) of $\Sigma^{f}$ --- however, it is
  related, after passing to presheaves, to the co-internalisation of
  $\funcs{E}{\infgrpds}$.
\end{notation}

Recall that in~\cite[Proposition 14.16. (2)]{chu21:_homot_segal}, an
algebraic pattern $\cat{P}$ is said to be \textbf{saturated} if the
inclusion $\cat{P}^{\elem}\hookrightarrow\cat{P}$ is
codense\footnote{It is written there as
  $\cat{P}^{\inrt}\hookrightarrow\cat{P}$ being codense, but the proof
  of Proposition 14.20 immediately confirms this as a typo. In
  addition, extendability is required as part of the definition, but
  it is not necessary for this characterisation.}. If the relevant
limits are co-Van Kampen, taking ``global sections'' of the
co-internalisation $(\cat{P}^{\elem}_{P/})_{-/}$ produces an
equivalence
$\colim_{E\in\op{(\cat{P}^{\elem}_{P/})}}
\cat{P}^{\elem}_{E/}\simeq\cat{P}^{\elem}_{P/}$ which we may thus
think of as \textbf{global saturation} for the pattern $\cat{P}$. As
has been observed independently by Thomas Blom and Jan Steinbrunner
(who graciously provided the following proof), this weaker property
turns out to be satisfied by every algebraic pattern, regardless of
saturation:

\begin{pros}
  \label{pros:globl-satur}
  Any algebraic pattern $\cat{P}$ is globally saturated, that is: for
  any $P\in\cat{P}$ the canonical map
  \begin{equation}
    \label{eq:strong-saturation-def}
    \colim_{E\in\op{(\cat{P}^{\elem}_{P/})}}
    \cat{P}^{\elem}_{E/}\to\cat{P}^{\elem}_{P/}
  \end{equation}
  induced by the $\Sigma^{u}$ for each $u\colon P\rightarrowtail E$ is
  an equivalence.
\end{pros}

\begin{proof}[Proof \cite{steinebr25:_global_satur}]
  All the $\infty$-categories appearing in the colimit can be
  rewritten as slices
  $\cat{P}^{\elem}_{E/}\simeq(\cat{P}^{\elem}_{P/})_{E/}$ of
  $\cat{P}^{\elem}_{P/}$.

  Now for any $(\infty,1)$-category $\cat{E}$, the colimit of the
  co-evaluation functor $\colim_{E\in\op{\cat{E}}}\cat{E}_{E/}$ can be
  computed in two steps: first take its lax colimit, which is its
  Grothendieck construction $\ev_{0}\colon\cat{Ar}(\cat{E})\to\cat{E}$
  of~\cite[Corollary 2.4.7.11]{lurie09:_higher}, and then rectify by
  localising at the $\ev_{0}$-cartesian arrows. By~\cite[Lemma
  2.4.7.5]{lurie09:_higher}, these are precisely those whose image
  under $\ev_{1}$ is an equivalence. Thus, taking
  $\cat{E}=\cat{P}^{\elem}_{P/}$, we need only exhibit $\ev_{1}$ as
  such a localisation, which it is by~\cite[Proposition
  7.1.12]{cisinski19:_higher_categ_homot_algeb} because it is a
  cocartesian fibration (whence smooth) whose fibres, admitting
  terminal objects, are contractible.
\end{proof}

This argument establishing global saturation for all algebraic
patterns is somewhat inexplicit, relying on technical properties of
localisation functors, so for many cases it can be instructive to
understand the property through more direct examination of the
pattern.

\begin{exmp}
  \cite[Proposition 5.13]{haugseng18:_iterat} shows explicitly that
  the algebraic pattern ${\op{\bbDelta}}^{\natural}$ for internal
  categories is globally saturated.
\end{exmp}

\section{Global saturation for the cell category $\Theta$}
\label{sec:globl-satur-cell}

The main example of interest for applying the result on general Segal
objects we will establish in~\cref{sec:monads-spans} is that of
(internal) $\omega$-categories. In order to provide a good
understanding of it for readers less familiar with Joyal's cell
category, we will here study ``by hand'' the global saturation of the
relevant algebraic pattern.

\begin{construct}
  Recall that the (non-reflexive) \textbf{globe category} $\mathbb{G}$
  is generated by objects $\overline{n}$, for all $n\in\mathbb{N}$,
  and arrows $i_{n}^{\pm}\colon\overline{n}\to\overline{n+1}$, as
  presented in the graph
  \begin{equation}
    \label{eq:globe-cat-generat}
    \begin{tikzcd}
      \overline{0} \arrow[r,shift left,"i_{0}^{+}"] \arrow[r,shift
      right,"i_{0}^{-}"'] & \overline{1} \arrow[r,shift
      left,"i_{1}^{+}"] \arrow[r,shift right,"i_{1}^{-}"'] & \cdots
      \arrow[r,shift left,"i_{n-1}^{+}"] \arrow[r,shift
      right,"i_{n-1}^{-}"'] & \overline{n} \arrow[r,shift
      left,"i_{n}^{+}"] \arrow[r,shift right,"i_{n}^{-}"'] &
      \cdots\text{,}
    \end{tikzcd}
  \end{equation}
  with the relations
  $i_{n+1}^{+}i_{n}^{\varepsilon}=i_{n+1}^{-}i_{n}^{\varepsilon}$ for
  any $n\in\mathbb{N}$ and any $\varepsilon\in\{+,-\}$. A
  \textbf{globular object} in an $(\infty,1)$-category $\cat{C}$ is a
  $\cat{C}$-valued presheaf on $\mathbb{G}$. A \textbf{strict
    $\omega$-category} is a globular set equipped with units and
  composition operations satisfying certain equations (spelled out for
  example in~\cite[p. 300]{street00}); such structure is monadic over
  $\funcs{\op{\mathbb{G}}}{Set}$, with monad $\func{F}\!_{\omega}$.

  The \textbf{cell category} $\Theta$ (first introduced
  in~\cite{joyal97:_disks_theta}) has as objects the globular sets
  that are pastings of appropriately composable globes --- a condition
  encoded precisely as the notion of globular sums in the sense
  of~\cite[\S{}2.1.1]{ara10:_sur_groth}
  or~\cite[\S{}1.1.2.2]{loubaton23:_theor} --- and as morphisms the
  morphisms of strict $\omega$-categories between their associated
  free $\omega$-categories. A morphism $f$ is \textbf{inert} (also
  called an immersion) if it is the image by $\func{F}\!_{\omega}$ of
  a morphism of globular sets, and \textbf{active} if in any
  factorisation $f=ia$ with $i$ inert, $i$ must be an identity
  (by~\cite[Proposition 3.3.11]{ara10:_sur_groth}, they correspond to
  the maps also known as algebraic, or covers). By~\cite[Lemma
  1.11]{berger02:_cellul_nerve_higher_categ} or~\cite[Proposition
  3.3.10]{ara10:_sur_groth}, the classes of active and inert morphisms
  form a unique factorisation system, so in particular an orthogonal
  one, on $\Theta$.
\end{construct}

\begin{notation}[Generic $n$-cells]
  For any $n\in\mathbb{N}$, the representable presheaf
  $\mathbb{G}(-,\overline{n})$ is canonically endowed with a structure
  of strict $\omega$-category (which comes from viewing it as the
  restriction to $\op{\mathbb{G}}$ of $\Theta(-,\overline{n})$). We
  denote this $\omega$-category $\cat{D}_{n}$; it is known as the
  \textbf{$n$-globe}, or as the generic (or ``walking'') $n$-cell.
\end{notation}

\begin{definition}
  \label{def:cell-pattrn-theta}
  The algebraic pattern ${\op{\Theta}}^{\natural}$ is the category
  $\op{\Theta}$, endowed with the inert--active factorisation system
  described above, and with elementary objects the $\ell$-globes
  (so that ${\op{\Theta}}^{\natural,\elem}\simeq\op{\mathbb{G}}$).
\end{definition}

It is an immediate consequence of the definition (and of the fact that
all inert maps into a globe in $\Theta$ also have to be from a globe)
that Segal ${\op{\Theta}}^{\natural}$-objects are exactly what are
called $\Theta$-models in~\cite{berger02:_cellul_nerve_higher_categ}.

\begin{lemma}[{\cite[Proposition 2.3.18]{ara10:_sur_groth}}]
  \label{lemma:theta-saturated}
  The pattern ${\op{\Theta}}^{\natural}$ is saturated.
\end{lemma}

\begin{proof}
  This is essentially a consequence of the definition of globular
  sums: any such globular set $T$ can be written as an iterated
  pushout
  $T\simeq\cat{D}_{i_{1}}\amalg_{\cat{D}_{i^{\prime}_{1}}}
  \dots\amalg_{\cat{D}_{i^{\prime}_{p-1}}}\cat{D}_{i_{p}}$, and
  by~\cite[Lemme 2.3.22]{ara10:_sur_groth}, the immersions
  $\cat{D}_{i}\rightarrowtail T$ featuring in this pushout define a cofinal
  subcategory of $\Theta^{\natural,\elem}_{/T}$.
\end{proof}

It follows from this that the definition of Segal
${\op{\Theta}}^{\natural}$-objects in a complete $(\infty,1)$-category
$\cat{C}$ coincides with that of (weak) $\omega$-categories in
$\cat{C}$ in the sense of~\cite{loubaton23:_theor}, albeit without the
Rezk-completeness (or univalent completeness) condition --- so that,
to be more precise, they correspond to flagged $\omega$-categories as
in~\cite{ayala18:_flagg}.

\begin{remark}
  \label{remark:theta-finite-dim}
  For any $\ell\in\mathbb{N}\cup\{\omega\}$, we let
  \begin{equation}
    \label{eq:theta-ell-def}
    \Theta_{\ell}=\Theta\cap\inflcats{\ell}
  \end{equation}
  be the $\ell$-dimensional cell category used in~\cite{rezk10:_cart};
  the pattern structure of~\Cref{def:cell-pattrn-theta} restricts to
  one on $\op{\Theta_{\ell}}$ whose Segal objects are internal flagged
  $\ell$-categories (and we obviously have $\Theta_{\omega}=\Theta$).

  There is an obvious filtration
  ${\op{\Theta_{1}}}^{\natural}\simeq{\op{\bbDelta}}^{\natural}
  \hookrightarrow{\op{\Theta_{2}}}^{\natural}\hookrightarrow
  {\op{\Theta_{3}}}^{\natural}\hookrightarrow\cdots$ and we recover
  ${\op{\Theta}}^{\natural}$ as its colimit. In particular, since each
  ${\op{\Theta_{\ell}}}^{\natural}$ for $\ell$ finite is known to be
  sound from~\cite[Example 3.3.18]{barkan22:_envel_algeb_patter}, it
  follows from~\Cref{lemma:colim-sound-pttrns} that
  ${\op{\Theta}}^{\natural}$ is sound as well.
\end{remark}

\begin{construct}
  Since the cells in a pasting diagram are unlabelled, the standard
  representation of objects of $\Theta$ contains redundant
  information. A more minimal presentation, suggested
  by~\cite{batanin98:_monoid_globul_categ_as_natur} and developed more
  thoroughly in~\cite{berger02:_cellul_nerve_higher_categ}
  and~\cite{ara10:_sur_groth}, of these objects is as level trees,
  functors from some $\op{[\ell]}$ ($\ell$ being the categorical
  dimension) to $\bbDelta$ whose value at the terminal object $0$ is
  $[0]$: the cells in the corresponding pasting diagram can be all
  recovered as the sectors in the tree.

  This description makes it easier to get a handle on the structure of
  these trees and their categories of inert morphisms: for a tree
  $T\colon\op{[\ell]}\to\bbDelta$, for $k\leq\ell$, we set
  $\uly{T}(k)$ to be the reunion, over $i\in T(k)$, of the
  $T(k+1)_{i}+1$, where $T(k+1)_{i}$ is the fibre of $T(k+1)\to T(k)$
  at $i$ (and where we decreed $T(\ell+1)$ to be
  $[-1]=\varemptyset$). Note that the assignment
  $\op{[\ell]}\ni k\mapsto\uly{T}(k)$ is \emph{not} functorial;
  however $\op{\mathbb{G}_{\leq\ell}}\ni\cat{D}_{k} \mapsto\uly{T}(k)$
  can be made functorial.
\end{construct}

\begin{remark}
  \label{remark:sectors-and-cells}
  The objects of $\uly{T}(k)$ can be understood as the
  \textbf{sectors} at level $k$ as defined
  by~\cite{berger02:_cellul_nerve_higher_categ} (and, likewise, their
  ordering is the natural left-to-right ordering of sectors in each
  fibre), so that $\uly{T}$ coincides with the globular set denoted
  $T^{\ast}$ in~\cite{batanin98:_monoid_globul_categ_as_natur}.

  In the dictionary between globular sums and trees, it is the sectors
  of a tree that correspond to the cells of the corresponding globular
  sum.
\end{remark}

We will now use the decompositions provided by the proof
of~\Cref{lemma:theta-saturated} to understand the categories
$\Theta^{\natural,\elem}_{/T}$.

\begin{lemma}
  \label{lemma:descript-elem-slices-theta}
  Let $T\in\Theta$ be any globular sum. Then
  $\Theta^{\natural,\elem}_{/T}$ is equivalent to the Grothendieck
  construction of the globular set $\uly{T}$.
\end{lemma}

\begin{proof}
  We will exhibit an explicit isomorphism between $\uly{T}$ and the
  underlying $\Theta^{\natural}$-graph of the presheaf represented by
  $T$, since then its Grothendieck construction is indeed the
  slice. Consider an object of $\Theta^{\natural,\elem}_{/T}$, given
  by a map $\cat{D}_{i}\rightarrowtail T$, \emph{i.e.} an element of
  $T(\cat{D}_{i})$. Since $\cat{D}_{i}$ is the free $i$-cell, this map
  is uniquely characterised by a choice of an $i$-cell in $T$. In
  terms of the associated trees, $D_{i}$ is a linear tree and so such
  a map is characterised by a choice of a branch at level $i$ and a
  sector around its top point. It follows
  from~\Cref{remark:sectors-and-cells} that these are exactly counted
  by the elements of $\uly{T}(\cat{D}_{i})$.
\end{proof}

\begin{exmp}
  For any elementary $\cat{D}_{\ell}$, the category
  $\Theta^{\natural,\elem}_{/\cat{D}_{\ell}}$ is freely generated by the
  graph
  \begin{equation}
    \label{eq:theta-elem-slice-over-globe}
    \begin{tikzcd}[row sep=small,column sep=tiny]
      & \characmap{C_{\ell}} & \\
      \characmap{C_{\ell-1}^{-}} \arrow[ur] & &
      \characmap{C_{\ell-1}^{+}} \arrow[ul] \\
      {} \arrow[u] \arrow[urr] & & {} \arrow[u]
      \arrow[ull,crossing over] \\
      {} \arrow[u,phantom,"\vdots"description] & & {}
      \arrow[u,phantom,"\vdots"description] \\
      \characmap{C_{1}^{-}} \arrow[u] \arrow[urr] & &
      \characmap{C_{1}^{+}} \arrow[u] \arrow[ull,crossing over] \\
      \characmap{C_{0}^{-}} \arrow[u] \arrow[urr] & &
      \characmap{C_{0}^{+}} \arrow[u] \arrow[ull,crossing over]
    \end{tikzcd}
  \end{equation}
  where we recall that $\cat{D}_{\ell}$ has a unique $\ell$-cell
  $C_{\ell}$ and, for any $0\leq i<\ell$, two $i$-cells $C_{i}^{\pm}$
  serving has source and target for the higher cells, and
  $\characmap{C_{i}^{\pm}}\colon\cat{D}_{i}\to\cat{D}_{\ell}$ denotes
  the (inert) map selecting the corresponding cell. In other words,
  $\Theta^{\natural,\elem}_{/\cat{D}_{\ell}}$ is the free-living
  $\ell$-iterated cospan, so that the category we are ultimately
  interested in, ${\op{\Theta}}^{\natural,\elem}_{\cat{D}_{\ell}/}$,
  which is its opposite, will be the free-living $\ell$-iterated
  span.
\end{exmp}

\begin{lemma}
  \label{lemma:slice-theta-binary-pushout}
  Let $T$ be a globular sum of the form
  $\cat{D}_{m}\amalg_{\cat{D}_{\ell}}\cat{D}_{n}$. Then
  $\Theta^{\natural,\inrt}_{/T}$ is the strict pushout of
  $1$-categories
  $\Theta^{\natural,\inrt}_{/\cat{D}_{m}}\amalg_{\Theta^{\natural,\inrt}_{/\cat{D}_{\ell}}}
  \Theta^{\natural,\inrt}_{/\cat{D}_{n}}$
\end{lemma}

\begin{proof}
  Let us call $E_{i}^{\varepsilon}$ the cells of $T$ in $\cat{D}_{m}$,
  $F_{i}^{\varepsilon}$ those in $\cat{D}_{n}$, and
  $C_{i}^{\varepsilon}$ those in $\cat{D}_{\ell}$, so that we have
  $E_{i}^{\varepsilon}=C_{i}^{\varepsilon}=F_{i}^{\varepsilon}$ for
  $i<\ell$ and $E_{\ell}^{+}=C_{\ell}=F_{\ell}^{-}$. The matter is then
  that of enumerating the cells and their relations, for which no
  listing can be as clear as simply drawing a generating graph:
  \begin{equation}
    \label{eq:theta-elem-slice-pushout}
    \begin{tikzcd}[row sep=small,column sep=tiny]
      & E_{m} & & & & F_{n} & \\
      E_{m-1}^{-} \arrow[ur] & & E_{m-1}^{+} \arrow[ul] & &
      F_{n-1}^{-} \arrow[ur] & & F_{n-1}^{+} \arrow[ul] \\
      {} \arrow[u] \arrow[urr] & & {} \arrow[u] \arrow[ull,crossing
      over] & & {}\arrow[u] \arrow[urr] & & {} \arrow[u]
      \arrow[ull,crossing over] \\
      {} \arrow[u,phantom,"\vdots"description] & & {}
      \arrow[u,phantom,"\vdots"description] & & {}
      \arrow[u,phantom,"\vdots"description] & & {}
      \arrow[u,phantom,"\vdots"description] \\
      E_{\ell+1}^{-} \arrow[u] \arrow[urr] & & E_{\ell+1}^{+}
      \arrow[u] \arrow[ull,crossing over] & & F_{\ell+1}^{-} \arrow[u]
      \arrow[urr] & & F_{\ell+1}^{+} \arrow[u] \arrow[ull,crossing over] \\
      E_{\ell}^{-} \arrow[u] \arrow[urr] & & & C_{\ell} \arrow[ul]
      \arrow[ulll,crossing over] \arrow[ur] \arrow[urrr] & & &
      F_{\ell}^{+} \arrow[u]
      \arrow[ull,crossing over] \\
      & C_{\ell-1}^{-} \arrow[urr] \arrow[ul] \arrow[urrrrr] & & & &
      C_{\ell-1}^{+}
      \arrow[ull,crossing over] \arrow[ur] \arrow[ulllll,crossing
      over] & \\
      & {} \arrow[u] \arrow[urrrr] & & & & {} \arrow[u]
      \arrow[ullll,crossing over] & \\
      & {} \arrow[u,phantom,"\vdots"description] & & & & {}
      \arrow[u,phantom,"\vdots"description] & \\
      & C_{0}^{-} \arrow[u] \arrow[urrrr] & & & & C_{0}^{+}\text{.}
      \arrow[u]
      \arrow[ullll,crossing over] &
    \end{tikzcd}
  \end{equation}
  Since the Grothendieck construction takes colimits of presheaves (of
  sets) to strict colimits of $1$-categories, and our slices are
  categories of elements as
  in~\Cref{lemma:descript-elem-slices-theta}, one can indeed recognise
  in~\cref{eq:theta-elem-slice-pushout} a strict pushout of three
  versions of~\cref{eq:theta-elem-slice-over-globe}.
\end{proof}

\begin{pros}
  The algebraic pattern ${\op{\Theta}}^{\natural}$ is globally
  saturated.
\end{pros}

\begin{proof}
  Again, we can use the decomposition
  $T\simeq\cat{D}_{i_{1}}\amalg_{\cat{D}_{i^{\prime}_{1}}}
  \dots\amalg_{\cat{D}_{i^{\prime}_{p-1}}}\cat{D}_{i_{p}}$ since it is
  cofinal, so that all we have to prove is that
  \begin{equation}
    \label{eq:7}
    \Theta^{\natural,\elem}_{/T}
    \simeq\Theta^{\natural,\elem}_{/\cat{D}_{i_{1}}}
    \mathbin{\mathop{\amalg}\limits_{\Theta^{\natural,\elem}_{/\cat{D}_{i^{\prime}_{1}}}}}
    \dots\mathbin{\mathop{\amalg}_{\Theta^{\natural,\elem}_{/\cat{D}_{i^{\prime}_{p-1}}}}}
    \Theta^{\natural,\elem}_{/\cat{D}_{i_{p}}}\text{.}
  \end{equation}
  To compute this pushout of $(\infty,1)$-categories, we will use the
  Joyal model structure for quasicategories. Letting
  $N_{\bullet}\cat{C}$ denote the nerve of an $(\infty,1)$-category
  $\cat{C}$, it is clear --- since $i^{\prime}_{j\pm1}<i_{j}$ for all
  $j$ in the decomposition --- that the maps of quasicategories
  $N_{\bullet}\Theta^{\natural,\elem}_{/\cat{D}_{i^{\prime}_{j\pm1}}}\to
  N_{\bullet}\Theta^{\natural,\elem}_{/\cat{D}_{i_{j}}}$ are injective
  in every degree, \emph{i.e.} cofibrations in the Joyal model
  structure, so that the pushout will coincide with the pushout of
  $1$-categories. The result for this strict pushout is then
  established \emph{via}~\Cref{lemma:slice-theta-binary-pushout}.
\end{proof}

\section{Generalised spans}
\label{sec:gener-mult}

For this~\namecref{sec:gener-mult}, we fix an algebraic pattern
$\cat{P}$ and a $\cat{P}$-complete $(\infty,1)$-category $\cat{C}$. We
will adapt to $\cat{P}$ the constructions and arguments
of~\cite[\S{}5]{haugseng18:_iterat}.

Recall that
$\cat{Ar}(\cat{P})\coloneqq\funcs{\mathbb{2}}{P}
\xrightarrow{\ev_{0}}\funcs{\mathbb{1}}{P}\simeq\cat{P}$ is a
cartesian fibration classifying the $\infty$-functor
$\cat{P}_{-/}\colon\op{\cat{P}}\to\infcats$. We let
$\cat{Ar}_{\inrt}(\cat{P})$ be the full sub-$(\infty,1)$-category of
$\cat{Ar}(\cat{P})$ on the inert arrows --- which, by the dual
of~\cite[Proposition 2.2.2]{barkan22:_envel_algeb_patter}, still
defines a cartesian fibration.

Our first goal is to show that
$\cat{Ar}_{\inrt}(\cat{P})\xrightarrow{\ev_{0}}\cat{P}$ classifies an
$\infty$-functor $\op{\cat{P}}\to\infcats$ whose action on objects is
$P\mapsto\cat{P}^{\inrt}_{P/}$.

\begin{construct}
  Since the factorisation system of $\cat{P}$ is functorial,
  projection onto the inert part of an arrow defines a functor
  $\func{inrt}\colon\funcs{\mathbb{2}}{P}
  \to\funcs{\mathbb{3}}{P}\to\funcs{\mathbb{2}}{P}$, which preserves
  the image of $\ev_{0}$ so defines a morphism of categories over
  $\funcs{\mathbb{1}}{P}$ (but not of cartesian fibrations over
  $\cat{P}$, as it does not preserve cartesian lifts of non-inert
  morphisms). We let $\func{inrt}(\cat{Ar}(\cat{P}))$ denote its
  essential image, whose objects are then the inert arrows of
  $\cat{P}$ while morphisms are the squares all of whose edges are
  inert --- so that, in particular, the fibre of
  $\restr{\ev_{0}}_{\func{inrt}(\cat{Ar}(\cat{P}))}$ at $P\in\cat{P}$
  is $(\cat{P}^{\inrt})_{P/}=\cat{P}^{\inrt}_{P/}$.
\end{construct}

\begin{lemma}
  \label{lemma:inert-transport}
  Consider a commuting triangle of inert arrows below-left
  \begin{equation}
    \label{eq:factoris-in-slice}
    \begin{tikzcd}[row sep=tiny,column sep=large]
      & Q \arrow[dd,tail,"h"] \\
      P \arrow[ur,tail,"g"] \arrow[dr,tail,"g^{\prime}"'] & \\
      & Q^{\prime}
    \end{tikzcd}\qquad\qquad
    \begin{tikzcd}[column sep=large]
      O \arrow[d,equals] \arrow[r,tail,"\func{inrt}(gf)"']
      \arrow[rr,bend left,"\Sigma^{f}g=gf"] & M
      \arrow[r,rightsquigarrow] \arrow[d,"\func{inrt}(\Sigma^{f}h)"] &
      Q
      \arrow[d,tail,"h"] \\
      O \arrow[r,tail,"\func{inrt}(g^{\prime}f)"] \arrow[rr,bend
      right,"\Sigma^{f}g^{\prime}=hgf"'] & M^{\prime}
      \arrow[r,rightsquigarrow] & Q^{\prime}
    \end{tikzcd}
  \end{equation}
  defining a morphism in $\cat{P}^{\inrt}_{P/}$, and let
  $O\xrightarrow{f}P$ be any arrow of $\cat{P}$, with inert--active
  factorisation of $\Sigma^{f}h$ as above-right. Then
  $\func{inrt}(\Sigma^{f}h)$ is inert.
\end{lemma}

\begin{proof}
  This is a direct application of the left-cancellability property for
  the left class of an orthogonal factorisation system (see for
  example~\cite[Proposition 5.2.8.6. (4)]{lurie09:_higher}
  or~\cite[Proposition 4.1.2.12]{loubaton23:_theor}).
\end{proof}

\begin{corlr}
  The projection $\func{inrt}(\cat{Ar}(\cat{P}))\to\cat{P}$ is a
  cartesian fibration, and coincides with
  $\cat{Ar}_{\mathrm{inrt}}(\cat{P})\to\cat{P}$. \qed{}
\end{corlr}

We thus obtain an $\infty$-functor
$\cat{P}^{\inrt}_{-/}\colon\op{\cat{P}}\to\infcats$ (whose restriction
to $\op{(\cat{P}^{\inrt})}$ is the co-internalisation of
$\cat{P}^{\inrt}$).

\begin{definition}
  We denote
  $\overline{\func{p}}\colon\overline{\cat{Span}}_{\cat{P}}(\cat{C})
  \to\cat{P}$ the cocartesian fibration classifying the
  $\infty$-functor
  $\funcs{P^{\inrt}_{-/}}{C}\colon\cat{P}\to\infcats$.
\end{definition}

Recall that by~\cite[Proposition 2.37]{barkan22:_arity_approx_operad},
for any $P\in\cat{P}$ there is an algebraic pattern structure on the
slice $\cat{P}_{P/}$, where an object (resp. an arrow) is elementary
(resp. inert, resp. active) if and only if its image by $\ev_{1}$ is
so in $\cat{P}$. Furthermore, by~\cite[Proposition 2.14 and
Proposition 2.4]{kositsyn21:_compl}, it restricts to an algebraic
pattern structure on $\cat{P}^{\inrt}_{P/}$ (which has no non-trivial
active morphisms).

\begin{definition}
  We call $\cat{Span}_{\cat{P}}(\cat{C})$ the full
  sub-$(\infty,1)$-category of
  $\overline{\cat{Span}}_{\cat{P}}(\cat{C})$ on the objects
  $(P,\func{F}\colon\cat{P}^{\inrt}_{P/}\to\cat{C})$ such that
  $\func{F}$ is a Segal $\cat{P}^{\inrt}_{P/}$-object.
\end{definition}

\begin{remark}
  An alternate construction of $\cat{Span}_{\cat{P}}(\cat{C})$ is
  provided by~\cite[Corollary 2.16]{kositsyn21:_compl}.
  
  We let
  $\func{i}^{\inrt}_{P}\colon\cat{P}^{\inrt}_{P/}\to\cat{P}_{P/}$
  denote the canonical inclusion (induced under slicing by
  $\cat{P}^{\inrt}\hookrightarrow\cat{P}$). By~\cite[Proposition
  2.15]{kositsyn21:_compl} (which is formulated in the case of
  $\cat{P}={\op{\bbDelta}}^{\natural}$ but only relies on the
  factorisation system), for any arrow $f\colon O\to P$ in $\cat{P}$,
  the induced $\infty$-functor
  $\Sigma^{f,\ast}\colon\funcs{P_{\mathnormal{O}/}}{C}
  \to\funcs{P_{\mathnormal{P/}}}{C}$ sends the image of
  $\func{i}^{\inrt}_{O,!}$ into the image of $\func{i}^{\inrt}_{P,!}$.

  We can then let
  $\overline{\cat{PreSpan}}_{\cat{P}}(\cat{C})\to\cat{P}$ denote the
  Grothendieck construction of the $\infty$-functor
  $\funcs{P_{-/}}{C}\colon\cat{C}\to\infcats$, and
  $\overline{\cat{Span}}_{\cat{P}}(\cat{C})$ is the full
  sub-$(\infty,1)$-category of
  $\overline{\cat{PreSpan}}_{\cat{P}}(\cat{C})$ on those objects
  $(P,\func{F}\colon\cat{P}_{P/}\to\cat{C})$ such that $\func{F}$ is
  in the image of $\func{i}^{\inrt}_{P,!}$ (so that it is determined
  by its restriction $\cat{P}^{\inrt}_{P/}\to\cat{C}$).
\end{remark}

\begin{lemma}
  \label{lemma:spans-are-cocart-fib}
  Assume the algebraic pattern $\cat{P}$ is sound. The restricted
  projection
  $\func{p}\colon\cat{Span}_{\cat{P}}(\cat{C})
  \hookrightarrow\overline{\cat{Span}}_{\cat{P}}(\cat{C})
  \xrightarrow{\overline{\func{p}}}\cat{P}$ is a cocartesian
  fibration.
\end{lemma}

\begin{proof}
  As explained in the proof of~\cite[Corollary
  5.12]{haugseng18:_iterat}, since $\cat{Span}_{\cat{P}}(\cat{C})$ is
  a full sub-$(\infty,1)$-category of
  $\overline{\cat{Span}}_{\cat{P}}(\cat{C})$, all we need to do is
  check that if $(P,\func{F})\to(Q,\func{G})$ is a
  $\overline{\func{p}}$-cocartesian morphism in
  $\overline{\cat{Span}}_{\cat{P}}(\cat{C})$ such that $\func{F}$ is
  Segal, then $\func{G}$ is Segal as well. Note also that such a
  cocartesian morphism consists of an arrow $f\colon P\to Q$ in
  $\cat{P}$ with
  $\func{G}\simeq\Sigma^{f,\ast}\func{F}=\func{F}\circ(\Sigma^{f})$:
  in other words, we must show that Segal objects are preserved by
  composition with codependent coproduct. That is, if
  $\func{F}\colon\cat{P}^{\inrt}_{P/}\to\cat{C}$ satisfies the Segal
  condition, then for every $g\colon Q\rightarrowtail Q^{\prime}$ we
  must have
  \begin{equation}
    \label{eq:segal-inert-pullback}
    \func{F}\bigl(\func{inrt}(P\xrightarrow{f}Q
    \overset{g}{\rightarrowtail}Q^{\prime})\bigr)
    \xrightarrow{\simeq}
    \lim_{(Q\overset{hg}{\rightarrowtail}E)\in(\cat{P}^{\inrt}_{Q/})^{\elem}_{g/}}
    \func{F}\bigl(\func{inrt}(
    P\xrightarrow{f}Q\overset{hg}{\rightarrowtail}
    E)\bigr)\text{.}
  \end{equation}

  By the functoriality of the construction $\Sigma^{(-),\ast}$ and the
  fact that active and inert morphisms for a factorisation system, we
  only need to check the comparison in the cases where $f$ is either
  purely inert or purely active. If $f$ is inert, then $\func{inrt}$
  acts as the identity on the composition, and the functor
  $\Sigma^{f,\ast}$ is even iso-Segal (since both
  $(\cat{P}^{\inrt}_{Q/})^{\elem}_{g/}$ and
  $(\cat{P}^{\inrt}_{P/})^{\elem}_{gf/}$ are then equivalent to
  $\cat{P}^{\elem}_{Q^{\prime}/}$ thanks to all maps being inert),
  which is strictly stronger than preserving Segal objects. The case
  of $f$ being active is where the soundness assumption comes into
  play.
  
  Write
  $\begin{tikzcd} P \arrow[r,tail,"{\func{inrt}(hgf)}"] & M_{hg}
    \arrow[r,rightsquigarrow] & E
  \end{tikzcd}$ the inert-active factorisation of the composition
  $hgf\colon P\rightsquigarrow Q\rightarrowtail E$, for any
  $Q\rightarrowtail E$ in $(\cat{P}^{\inrt}_{Q/})^{\elem}_{g/}$ given
  by $h\colon Q^{\prime}\rightarrowtail E$. Then, since $\func{F}$ is
  Segal, the right-hand side limit in~\cref{eq:segal-inert-pullback}
  becomes a double limit
  \begin{equation}
    \label{eq:dbl-decomp-segal-inrt}
    \lim_{(Q\overset{hg}{\rightarrowtail}E)\in(\cat{P}^{\inrt}_{Q/})^{\elem}_{g/}}
    \lim_{(M_{hg}\rightarrowtail E^{\prime})\in(\cat{P}^{\inrt}_{P/})^{\elem}_{M_{hg}/}}
    \func{F}(P\rightarrowtail M_{hg}\rightarrowtail E^{\prime})\text{,}
  \end{equation}
  summarised by the dashed arrows in the diagram
  \begin{equation}
    \label{eq:dbl-lim-indexing}
    \begin{tikzcd}
      P \arrow[d,rightsquigarrow,"f"'] \arrow[r,tail] & M_{g}
      \arrow[d,rightsquigarrow] \arrow[r,tail,"\func{inrt}(h)"] &
      M_{hg}
      \arrow[d,rightsquigarrow] \arrow[r,tail,dashed] &
      E^{\prime}\text{,} \\
      Q \arrow[r,tail,"g"'] & Q^{\prime} \arrow[r,tail,"h"'] & E &
    \end{tikzcd}
  \end{equation}
  where the inert transport arrow $\func{inrt}(h)$ comes
  from~\cref{lemma:inert-transport}. In particular, we can see that
  $(\cat{P}^{\inrt}_{Q/})^{\elem}_{g/}
  \simeq\cat{P}^{\elem}_{Q^{\prime}/}$ and
  $(\cat{P}^{\inrt}_{P/})^{\elem}_{M_{hg}/}
  \simeq\cat{P}^{\elem}_{M_{hg}/}$.

  Then, as explained in~\cite[Observation
  3.3.6]{barkan22:_envel_algeb_patter}, soundness of $\cat{P}$ allows
  this double limit to be computed as a limit over
  $(M_{g}\rightarrowtail E^{\prime})\in\cat{P}^{\elem}_{M_{g}/}$. But
  this is precisely what we obtain from the Segal decomposition for
  $\func{F}(P\rightarrowtail M_{g})$ in the left-hand side term
  of~\cref{eq:segal-inert-pullback}.
\end{proof}

\begin{pros}
  Let $\cat{P}$ be a sound algebraic pattern. The cocartesian
  fibration $\func{p}\colon\cat{Span}_{\cat{P}}(\cat{C})\to\cat{P}$ is
  a Segal fibration, that is the $\infty$-functor
  $\internalcat{Span_{\cat{P}}(\cat{C})} \colon\cat{P}\to\infcats$ it
  classifies defines a $\cat{P}$-monoidal $(\infty,1)$-category.
\end{pros}

\begin{proof}
  To make the definition explicit, we need to show that for any
  $P\in\cat{P}$,
  \begin{equation}
    \label{eq:segal-cond-span}
    \seg{\cat{P}^{\inrt}_{P/}}(\cat{C})
    \to\lim_{E\in\cat{P}^{\elem}_{P/}}\seg{\cat{P}^{\inrt}_{E/}}(\cat{C})
  \end{equation}
  is an equivalence. Since $\cat{P}^{\inrt}_{P/}$ only has inert
  morphisms, the right Kan extension $\infty$-functor
  \begin{equation}
    \label{eq:seg-inrt-equal-fctr-on-el}
    \seg{\cat{P}^{\elem}_{P/}}(\cat{C})
    \simeq\funcs{P^{\elem}_{\mathnormal{P/}}}{C}
    \to\seg{\cat{P}^{\inrt}_{P/}}(\cat{C})
  \end{equation}
  is an equivalence. Similarly, every factor
  $\seg{\cat{P}^{\inrt}_{E/}}(\cat{C})$ in~\cref{eq:segal-cond-span}
  is equivalent to $\funcs{P^{\elem}_{\mathnormal{E/}}}{C}$, and so
  the map of~\cref{eq:segal-cond-span} takes the form
  \begin{equation}
    \label{eq:segal-cond-span-unsegald}
    \funcs{P^{\elem}_{\mathnormal{P/}}}{C}
    \to\lim_{E\in\cat{P}^{\elem}_{P/}}\funcs{P^{\elem}_{\mathnormal{E/}}}{C}\text{.}
  \end{equation}
  By the property of global saturation from~\cref{pros:globl-satur},
  and since enriched homs (or cotensors) send colimits in the first
  variable to limits, this map is an equivalence.
\end{proof}

\section{$\cat{P}$-Monads in $\cat{P}$-spans}
\label{sec:monads-spans}

This~\namecref{sec:monads-spans} will follow very closely the
structure of~\cite[\S{}3]{haugseng21:_segal}.

\begin{lemma}
  The $\infty$-functor
  $\cat{Ar}_{\inrt}(\cat{P})\xrightarrow{\ev_{1}}\cat{P}$ admits a
  right adjoint right inverse.
\end{lemma}

\begin{proof}
  The functor $\characmap{1}\colon\mathbb{1}\to\mathbb{2}$ has a
  retraction $\mathbb{2}\xrightarrow{!_{\mathbb{2}}}\mathbb{1}$, which
  upgrades in fact to a left adjoint left inverse: we clearly have
  $!_{\mathbb{2}}\circ\characmap{1}=!_{\mathbb{1}}=\id_{\mathbb{1}}$,
  while there is a (unique, since $\mathbb{2}$ is posetal) natural
  transformation
  $\id_{\mathbb{2}}\Rightarrow\characmap{1}\circ!_{\mathbb{2}}=\const_{1}$,
  and it is easily checked (by unicity of $!$) that these two
  transformations satisfy the triangle identities.

  Now note that
  $\ev_{1}\colon\cat{Ar}(\cat{P})=\funcs{\mathbb{2}}{P}
  \to\funcs{\mathbb{1}}{P}$ is exactly given by
  $\funcs{\characmap{\mathnormal{1}}}{P}$, and so, as powering with
  $\cat{P}$ is $(\infty,2)$-functorial (that is, as an
  $\infty$-functor $\funcs{(-)}{P}\colon\op{\infcats}\to\infcats$, it
  is $\infcats$-linear, and so upgrades to an $(\infty,2)$-functor),
  it has a right adjoint right inverse given by
  $\funcs{!_{\mathbb{2}}}{P}$. The latter $\infty$-functor can be
  described very explicitly: it maps an object $P\in\cat{P}$ to its
  identity arrow $\id_{P}\in\cat{Ar}(\cat{P})$.

  In particular, it factors through $\cat{Ar}_{\inrt}(\cat{P})$ --- as
  identity arrows are inert --- and since this
  sub-$(\infty,1)$-category of $\cat{Ar}(\cat{P})$ is full, the
  astriction of $\funcs{!_{\mathbb{2}}}{P}$ to it furnishes the
  desired right adjoint right inverse to
  $\cat{Ar}_{\inrt}(\cat{P})\xrightarrow{\ev_{1}}\cat{P}$.
\end{proof}

Given its description, we will denote
$\characmap{\id}\colon\cat{P}\to\cat{Ar}_{\inrt}(\cat{P})$ the right
adjoint right inverse to $\ev_{1}$. The unit will simply be known as
$\eta\colon\id_{\cat{Ar}_{\inrt}(\cat{P})}\Rightarrow\characmap{\id}\circ\ev_{1}$;
its component at $(P\rightarrowtail Q)\in\cat{Ar}_{\inrt}(\cat{P})$ is
the square
\begin{equation}
  \label{eq:unit-component}
  \eta_{(P\rightarrowtail Q)}\colon\begin{tikzcd}
    P \arrow[d,tail] \arrow[r,tail] & Q \arrow[d,equals] \\
    Q \arrow[r,equals] & Q\text{.}
  \end{tikzcd}
\end{equation}

\begin{pros}
  \label{pros:localis-base}
  The $\infty$-functor
  $\cat{Ar}_{\inrt}(\cat{P})\xrightarrow{\ev_{1}}\cat{P}$ exhibits
  $\cat{P}$ as the localisation of $\cat{Ar}_{\inrt}(\cat{P})$ at the
  set $\mathscr{I}$ of $\ev_{0}$-cartesian morphisms lying over inert
  arrows of $\cat{P}$.
\end{pros}

\begin{proof}
  Let $\mathscr{W}$ be the set of morphisms in
  $\cat{Ar}_{\inrt}(\cat{P})$ inverted by $\ev_{1}$;
  by~\cite[Corollary 2.4.7.11 and Lemma 2.4.7.]{lurie09:_higher}
  (\emph{cf.} also~\cite[Proposition
  2.2.2.(2)]{barkan22:_envel_algeb_patter}), $W$ consists exactly of
  the $\ev_{0}$-cartesian morphisms, so that we do have
  $\mathscr{I}\subset\mathscr{W}$. If
  $(f,g)\colon(P\rightarrowtail Q) \to(P^{\prime}\rightarrowtail
  Q^{\prime})$ is a morphism in $\cat{Ar}_{\inrt}(\cat{P})$ lifting
  $f\colon P\to P^{\prime}$ in $\cat{P}$, it is in $\mathscr{W}$ if
  and only if $g\colon Q\to Q^{\prime}$ is an equivalence so that we
  have a commutative square
  \begin{equation}
    \label{eq:invert-saturated-cart-inrt}
    \begin{tikzcd}
      (P\rightarrowtail Q) \arrow[r,"{(f,g)}"] \arrow[d] &
      (P^{\prime}\rightarrowtail Q^{\prime}) \arrow[d] \\
      (Q=Q) \arrow[r,"\simeq"] & (Q^{\prime}=Q^{\prime})
    \end{tikzcd}
  \end{equation}
  in which the two vertical morphisms are in $\mathscr{I}$. Any
  $\infty$-functor from $\cat{Ar}_{\inrt}(\cat{P})$ to some
  $(\infty,1)$-category $\cat{C}$ inverting the morphisms in
  $\mathscr{I}$ will then send the square
  of~\cref{eq:invert-saturated-cart-inrt} to a square whose veritcal
  arrows (in addition to the lower horizontal one) are equivalences,
  whence its upper horizontal is one as well since equivalences always
  satisfy the 2-of-3 property. This means that such an
  $\infty$-functor automatically inverts all the morphisms in
  $\mathscr{W}$, and we only need to show that $\ev_{1}$ is a
  localisation, along $\mathscr{W}$. This follows readily from the
  fact that it has a right adjoint right inverse (in fact it is
  equivalent to it), but in our specific situation it can be seen in a
  more explicit way.

  Let $\cat{C}$ be again any $(\infty,1)$-category and let us consider
  the comparison $\infty$-functor
  $\funcs{\mathnormal{\ev_{1}}}{C}\colon\funcs{P}{C}
  \to\funcs{Ar_{\inrt}(P)}{C}_{(\mathscr{W})}$, where the target
  denotes the full sub-$(\infty,1)$-category of
  $\funcs{Ar_{\inrt}(P)}{C}$ on the $\infty$-functors inverting the
  morphisms in $\mathscr{W}$ (through which
  $\funcs{\mathnormal{\ev_{1}}}{C}$ does factor by definition of
  $\mathscr{W}$). The crux of the matter is that the components of the
  unit transformation $\eta$ all belong to $\mathscr{I}$ --- as can be
  seen in~\cref{eq:unit-component} --- and so \emph{a fortiori} to
  $\mathscr{W}$. Hence, the adjunction
  $\funcs{\mathnormal{\ev_{1}}}{C}
  \dashv\funcs{\mathnormal{\characmap{\id}}}{C}$ restricts on
  $\funcs{Ar_{\inrt}(P)}{C}_{(\mathscr{W})}$ to an equivalence (as its
  counit was already an identity, and its unit becomes one after this
  restriction), which means that $\ev_{1}$ is a localisation along
  $\mathscr{W}$.
\end{proof}

\begin{construct}
  Let $\func{f}\colon\cat{X}\to\cat{P}$ be an $\infty$-functor such
  that $\cat{X}$ admits $\func{f}$-cocartesian lifts of inert
  morphisms. Consider the solid pullback
  \begin{equation}
    \label{eq:hackerman}
    \begin{tikzcd}
      & \cat{Ar}_{\inrt}(\cat{P}) \arrow[dddl,bend left,"\ev_{0}"{near
        end},""{name=evoback,above}] & &
      \cat{X}\lmtimes_{\cat{P}}\cat{Ar}_{\inrt}(\cat{P})%
      \arrow[dddl,bend left,"\func{f}^{\ast}\ev_{0}"]%
      \arrow[ll,"\ev_{0}^{\ast}\func{f}"'] \\
      \cat{Ar}_{\inrt}(\cat{P})%
      \arrow[ur,bend left,dashed,start anchor={[xshift=2ex]north
        west},"\characmap{\id}\circ\ev_{1}",""{name=monad,below}]%
      \arrow[ur,bend right,start anchor={[xshift=1ex,yshift=-1ex]north
        east},"\id"',""{name=idAr,above}]%
      \arrow[from=idAr,to=monad,dashed,Rightarrow,shorten=6,"\eta"']%
      \arrow[dd,start anchor={[xshift=2ex]south west},bend
      right=5,"\ev_{0}"',""{name=evo,near end}]%
      \arrow[from=evo,to=evoback,dashed,Rightarrow,bend
      left=42,"\ev_{0}\eta"']%
      \arrow[from=rr,crossing over,"\ev_{0}^{\ast}\func{f}"{near
        start}]%
      & & \cat{X}\lmtimes_{\cat{P}}\cat{Ar}_{\inrt}(\cat{P})%
      \arrow[dd,"\func{f}^{\ast}\ev_{0}"']%
      \arrow[ddll,phantom,very near start,"\llcorner"]%
      \arrow[ur,bend right,"\func{f}^{\ast}\id=\id"'{very near
        start},""{name=idXAr,above}]%
      \arrow[ur,bend left,dotted,end anchor={[yshift=3ex]south
        west},""{name=liftdar,below}]%
      \arrow[from=idXAr,to=liftdar,Rightarrow,
      dotted,shorten=6,"\eta_{!}"]
      & \\
      & & & \\
      \cat{P} & & \cat{X} \arrow[ll,"\func{f}"] &
    \end{tikzcd}
  \end{equation}
  which is a (strongly) commutative diagram in the
  $(\infty,2)$-category $\infcats$. Adding
  $\characmap{\id}\circ\ev_{1}$, represented as a dashed arrow, the
  induced back-left triangle does not commute; however, adding as well
  the unit cell $\eta$ and its whiskering
  $\ev_{0}\eta\colon\ev_{0}=\ev_{0}\circ\id_{\cat{Ar}_{\inrt}(\cat{P})}
  \Rightarrow\ev_{0}\circ\characmap{\id}\circ\ev_{1}$ we obtain a
  ``2-commutative'' pasting diagram.

  Now as $\func{f}$ admits cocartesian lifts of inert arrows so does
  its base-change $\ev_{0}^{\ast}\func{f}$ (since cocartesian lifts
  are stable by pullback, by the co-dual of~\cite[Proposition
  5.2.4]{riehl22:_elemen}), and so, using the formulation of
  cocartesian lifts from~\cite[Definition 5.4.2]{riehl22:_elemen}, the
  transformation $\eta(\ev_{0}\func{f}^{\ast})$, whose components are
  inert, admits an $\ev_{0}\func{f}^{\ast}$-cocartesian dotted lift
  $\id_{\cat{X}\times_{\cat{P}}\cat{Ar}_{\inrt}(\cat{P})}
  \xRightarrow{\eta_{!}}
  (\id_{\cat{X}\times_{\cat{P}}\cat{Ar}_{\inrt}(\cat{P})})_{\eta}
  \eqqcolon\func{f}^{\eta}(\characmap{\id}\circ\ev_{1})$.

  We can finally define
  \begin{equation}
    \label{eq:def-fast-ev1}
    \func{f}^{\ast}\ev_{1} \coloneqq(\func{f}^{\ast}\ev_{0})
    \circ\func{f}^{\eta}(\characmap{\id}\circ\ev_{1})\text{.}
  \end{equation}  
  Explicitly, $\func{f}^{\ast}\ev_{1}$ sends an object
  $(X,\jmath\colon\func{f}X\rightarrowtail Q)
  \in\cat{X}\times_{\cat{P}}\cat{Ar}_{\inrt}(\cat{P})$ to
  $(X_{Q},Q=Q)$ where $X\xrightarrow{\jmath_{!}}X_{Q}$ is a
  cocartesian lift of $\jmath$. By construction it comes equipped with
  a natural transformation that we will call
  $\func{f}^{\ast}\alpha\coloneqq(\func{f}^{\ast}\ev_{0})\eta_{!}
  \colon\func{f}^{\ast}\ev_{0}\Rightarrow\func{f}^{\ast}\ev_{1}$,
  sitting in the diagram
  \begin{equation}
    \label{eq:lax-basechg-ev1}
    \begin{tikzcd}
      & \cat{P} \arrow[ddl,bend left,very near start,""'{name=idP}] &
      & \cat{X}\times_{\cat{P}}\cat{P} \simeq\cat{X} \arrow[ddl,bend
      left,"\func{f}^{\ast}\id_{\cat{P}}=\id_{\cat{X}}",""'{name=idX,very
        near
        start}] \arrow[ll] \\
      \cat{Ar}_{\inrt}(\cat{P}) \arrow[ur,"\ev_{1}"]
      \arrow[d,"\ev_{0}"',""{name=evo}]
      \arrow[from=evo,to=idP,Rightarrow,near start,"\alpha"']
      \arrow[from=rr,crossing over] & &
      \cat{X}\lmtimes_{\cat{P}}\cat{Ar}_{\inrt}(\cat{P})
      \arrow[d,"\func{f}^{\ast}\ev_{0}"',""{name=fevo}]
      \arrow[from=fevo,to=idX,shorten=10,Rightarrow,near
      start,"\func{f}^{\ast}\alpha"'] \arrow[dll,phantom,very near
      start,"\llcorner"] \arrow[ur,"\func{f}^{\ast}\ev_{1}"] & \\
      \cat{P} & & \cat{X} \arrow[ll,"\func{f}"] &
    \end{tikzcd}
  \end{equation}
  whose front and back squares are cartesian, but whose top square is
  not --- and where the natural transformation $\alpha$ comes from
  cotensoring with $\cat{P}$ the canonical $2$-cell
  $\characmap{0<1} \colon\characmap{0}\Rightarrow\characmap{1}
  \colon\mathbb{1}\to\mathbb{2}$ (in particular, it is easily checked
  that the adjunction $!_{\mathbb{2}}\dashv\characmap{1}$ lives under
  $\mathbb{1}$ so that $\ev_{1}\dashv\characmap{\id}$ lives over
  $\cat{P}$).  Conjecturally, the right face
  of~\cref{eq:lax-basechg-ev1} could be seen in terms of the
  $(\infty,3)$-topos of $(\infty,2)$-categories as the strong base
  change, along $\func{f}$ admitting enough cocartesian lifts, between
  fibrational lax slice $(\infty,2)$-categories, justifying our
  notation, though since the conditions for its construction are
  rather specific we will not pursue this point of view in further
  generality.
\end{construct}

\begin{lemma}
  The $\infty$-functor $\func{f}^{\ast}\ev_{1}$ admits a right adjoint
  right inverse.
\end{lemma}

\begin{proof}
  Note that in addition to being a right adjoint right inverse to
  $\ev_{1}$, the map $\characmap{\id}$ is also a left adjoint right
  inverse to $\ev_{0}$. We will denote the counit of this adjunction
  $\kappa$. Since the identity unit exhibits
  $\ev_{0}\circ\characmap{\id}=\id_{\cat{P}}$, the $\infty$-functor
  $\characmap{\id}$ lifts strongly to
  $\func{f}^{\ast}\characmap{\id}\colon\cat{X}
  \to\cat{X}\times_{\cat{P}}\cat{Ar}_{\inrt}(\cat{P})$: the equivalent
  of~\cref{eq:lax-basechg-ev1} with $\ev_{1}$ replaced by
  $\characmap{\id}$ (and $\alpha$ replaced by the identity unit,
  \emph{mutatis mutandis}) is a strongly commutative diagram, and
  fully cartesian. We claim that $\func{f}^{\ast}\characmap{\id}$ is
  the sought-after right adjoint right inverse to
  $\func{f}^{\ast}\ev_{1}$.

  To see this, we will show that the transformation $\eta_{!}$
  constructed in~\cref{eq:hackerman} works as a unit with identity
  counit; it requires first identifying its target
  $\func{f}^{\eta}(\characmap{\id}\circ\ev_{1})$ as
  $\func{f}^{\ast}\characmap{\id}\circ\func{f}^{\ast}\ev_{1}$. This is
  in fact trivial, because the transformations
  $\eta_{!}\colon\id\Rightarrow\func{f}^{\eta}(\characmap{\id}\circ\ev_{1})$
  and
  $(\func{f}^{\ast}\characmap{\id}) (\func{f}^{\ast}\ev_{0})
  \eta_{!}\colon\id\Rightarrow
  \func{f}^{\ast}\characmap{\id}\circ\func{f}^{\ast}\ev_{1}$ are both
  $\func{f}$-cocartesian lifts of
  $\eta\colon\id\Rightarrow\characmap{\id}\circ\ev_{1}$, but there is
  another interesting way of seeing it, that we detail in the next
  paragraph.

  Since the unit of the adjunction $\characmap{\id}\dashv\ev_{0}$ is
  an equivalence, the triangle identities imply that the whiskering
  $\kappa\characmap{\id}$ is the identity transformation of
  $\characmap{\id}$ 
  , and also
  $\kappa\characmap{\id}\ev_{1}\simeq\id_{\characmap{\id}\ev_{1}}$. There
  are now two things we can do: since
  $\id_{\func{f}^{\eta}(\characmap{\id}\ev_{1})}$ is a cocartesian
  lift of $\id_{\characmap{\id}\ev_{1}}$, the transformation
  $\id_{\func{f}^{\eta}(\characmap{\id}\ev_{1})}$ factors through a
  unique lift $\func{f}^{\eta}(\kappa\characmap{\id}\ev_{1})$ of
  $\kappa\characmap{\id}\ev_{1}$, which because of the factorisation
  has to be an identity. At the same time, one can take a cocartesian
  lift of $\kappa\characmap{\id}\ev_{1}$, which is easily seen to
  coincide with $\func{f}^{\eta}(\kappa\characmap{\id}\ev_{1})$; as a
  cocartesian lift of an identity, it is, again, an identity.  We thus
  have an equivalence
  \begin{equation}
    \label{eq:compos-lifts}
    \begin{split}
      (\func{f}^{\ast}\characmap{\id})\circ(\func{f}^{\ast}\ev_{1})
      =(\func{f}^{\ast}\characmap{\id})&\circ(\func{f}^{\ast}\ev_{0})
      \circ\func{f}^{\eta}(\characmap{\id}\circ\ev_{1})\\
      &\xrightarrow[\func{f}^{\eta}(\kappa\characmap{\id}\ev_{1})]{\simeq}
      \func{f}^{\eta}(\characmap{\id}\circ\ev_{1})\text{,}
    \end{split}
  \end{equation}
  expressing the decomposition we needed.
  
  Furthermore, constructing the equivalent of~\cref{eq:hackerman} but
  with $\ev_{1}\circ\characmap{\id}$ in place of $\id_{\cat{P}}$ (so
  with structure map to $\cat{P}$ given by $\id_{\cat{P}}$ instead of
  $\ev_{0}$), and with the identity counit
  $\varepsilon\colon\ev_{1}\circ\characmap{\id}
  \xRightarrow{=}\id_{\cat{P}}$ instead of $\eta$, we obtain, after
  strongly pulling back $\ev_{1}\circ\characmap{\id}$, an
  $\func{f}$-cocartesian transformation
  $\varepsilon_{!}\colon\func{f}^{\ast}(\ev_{1}\circ\characmap{\id})
  \Rightarrow(\func{f}^{\ast}(\ev_{1}\circ\characmap{\id}))_{\varepsilon}
  =\id_{\cat{X}}$, which as a cocartesian lift of $\varepsilon$ which
  is an identity, is itself an equivalence.

  Finally, the fact that $\func{f}^{\ast}\eta\coloneqq\eta_{!}$ and
  $\varepsilon_{!}$ satisfy the triangle identities is a consequence
  of the triangle identities for $\eta$ and $\varepsilon$, to which is
  applied the same reasoning we used to obtain~\cref{eq:compos-lifts}.
\end{proof}

It is worthwhile to note that the component of
$\func{f}^{\ast}\eta\colon\id\Rightarrow\func{f}^{\ast}\characmap{\id}\circ\func{f}^{\ast}\ev_{1}$
at an object
$(X,{\func{f}X\overset{\jmath}{\rightarrowtail}
  Q})\in\cat{X}\times_{\cat{P}}\cat{Ar}_{\mathrm{inrt}}(\cat{P})$ is
\begin{equation}
  \label{eq:basechgd-unit-component}
  \func{f}^{\ast}\eta_{(X,\func{f}X\rightarrowtail
    Q)}\colon\left(X,\begin{tikzcd}
    \func{f}X \arrow[d,tail] \\
    Q
  \end{tikzcd}\right)
  \begin{tikzcd}
    \vphantom{\func{f}X} \arrow[r,"{(X,\jmath)}"] & \vphantom{Q} \\
    \vphantom{Q} \arrow[r,"{(X,\id_{Q})}"'] & \vphantom{Q}
  \end{tikzcd}
  \left(X,\begin{tikzcd}
    Q \arrow[d,equals] \\
    Q
  \end{tikzcd}\right)\text{.}
\end{equation}

\begin{pros}
  \label{pros:localis-fibrtn}
  Let $\func{f}\colon\cat{X}\to\cat{P}$ be an $\infty$-functor such
  that $\cat{X}$ admits $\func{f}$-cocartesian lifts of inert
  morphisms. The $\infty$-functor
  $\func{f}^{\ast}\ev_{1}
  \colon\cat{X}\times_{\cat{P}}\cat{Ar}_{\inrt}(\cat{P}) \to\cat{X}$
  exhibits $\cat{X}$ as the localisation of
  $\cat{X}\times_{\cat{P}}\cat{Ar}_{\inrt}(\cat{P})$ at the set
  $\mathscr{I}_{\cat{X}}$ of morphisms
  $(X;(\func{f}(X)\rightarrowtail Q))
  \to(X^{\prime};(\func{f}(X^{\prime})\rightarrowtail Q^{\prime}))$
  such that
  \begin{itemize}
  \item $X\to X^{\prime}$ is $\func{f}$-cocartesian and
  \item
    $(\func{f}(X)\rightarrowtail Q)
    \to(\func{f}(X^{\prime})\rightarrowtail Q^{\prime})$ is
    $\ev_{0}$-cartesian and $\ev_{0}$-over an inert arrow.
  \end{itemize}
\end{pros}

\begin{proof}
  The proof follows the lines of that of~\Cref{pros:localis-base}. Let
  $\mathscr{W}_{\cat{X}}$ be the class of morphisms inverted by
  $\func{f}^{\ast}\ev_{1}$. A morphism of
  $\cat{X}\times_{\cat{P}}\cat{Ar}_{\inrt}(\cat{P})$, of the form
  $(\xi,\theta)$ where $\xi\colon X\to Y$ in $\cat{X}$ and $\theta$
  sits is a commutative square
  \begin{equation}
    \label{eq:morph-in-basechgd-arinrt}
    \begin{tikzcd}
      \func{f}X \arrow[d,tail,"\jmath"'] \arrow[r,"\func{f}\xi"] &
      \func{f}X^{\prime}
      \arrow[d,tail,"\jmath^{\prime}"] \\
      Q \arrow[r,"\theta"'] & Q^{\prime}
    \end{tikzcd}
  \end{equation}
  in $\cat{P}$, is in $\mathscr{W}_{\cat{X}}$ if and only if $\theta$
  is an equivalence $Q\simeq Q^{\prime}$, so that it induces a
  commutative square
  \begin{equation}
    \label{eq:invert-saturated-cart-inrt-basechgd}
    \begin{tikzcd}
      (X,\func{f}X\rightarrowtail Q) \arrow[r,"{(\xi,\theta)}"]
      \arrow[d,"\jmath_{!}"'] &
      (X^{\prime},\func{f}X^{\prime}\rightarrowtail
      Q^{\prime}) \arrow[d,"\jmath^{\prime}_{!}"] \\
      (X_{Q},\func{f}(X_{Q})=Q) \arrow[r,"{(\jmath_{!}\xi,\theta)}"']
      &
      (X^{\prime}_{Q^{\prime}},\func{f}(X^{\prime}_{Q^{\prime}}=Q^{\prime}))
    \end{tikzcd}
  \end{equation}
  where $X\xrightarrow{\jmath_{!}}X_{Q}$ and
  $X^{\prime}\xrightarrow{\jmath^{\prime}_{!}}X^{\prime}_{Q^{\prime}}$
  are cocartesian lifts of $\jmath$ and $\jmath^{\prime}$, and
  $\jmath_{!}\xi$ is the arrow $X_{Q}\to X^{\prime}_{Q^{\prime}}$
  uniquely induced by cartesianity, which is invertible since it lifts
  the isomorphism $Q\simeq Q^{\prime}$. The vertical morphisms are in
  $\mathscr{I}_{\cat{X}}$ by construction, so it follows from the
  2-of-3 property of equivalences that any $\infty$-functor that
  inverts the morphisms in $\mathscr{I}_{\cat{X}}$ will invert the
  morphisms in $\mathscr{W}_{\cat{X}}$, and that the localisations
  along $\mathscr{I}_{\cat{X}}$ and $\mathscr{W}_{\cat{X}}$ coincide.

  But again, it can be seen in~\cref{eq:basechgd-unit-component} that
  the components of $\func{f}^{\ast}\eta$ are in
  $\mathscr{I}_{\cat{X}}$ whence in $\mathscr{W}_{\cat{X}}$, so
  $\func{f}^{\ast}\ev_{1}$ is indeed a localisation along
  $\mathscr{W}_{\cat{X}}$.
\end{proof}

From now on, we assume that $\cat{P}$ is a sound pattern, and we let
$\cat{C}$ be a $\cat{P}$-complete $(\infty,1)$-category.

\begin{corlr}
  Let $\func{f}\colon\cat{X}\to\cat{P}$ be an $\infty$-functor such
  that $\cat{X}$ admits $\func{f}$-cocartesian lifts of inert
  morphisms. There is a fully faithful $\infty$-functor
  $\funcs{X}{C}
  \hookrightarrow\funcs{X}{\overline{Span}_{P}(C)}_{/\cat{P}}$ whose
  essential image is spanned by the $\infty$-functors preserving
  cocartesian morphisms over inert morphisms of $\cat{P}$.
\end{corlr}

\begin{proof}
  Direct application of~\cite[Proposition
  7.3]{gepner17:_lax_colim_free_fibrat_categ} shows that for any
  $(\infty,1)$-category $\cat{X}$ over $\cat{P}$ there is an
  equivalence
  \begin{equation}
    \label{eq:funcs-span-hom-fibrtn}
    \funcs{X}{\overline{Span}_{P}(C)}_{/\cat{P}}
    \simeq\funcs{X\times_{P}Ar_{\inrt}(P)}{C}\text{,}
  \end{equation}
  in which an $\infty$-functor
  $\func{S}\colon\cat{X}\to\overline{\cat{Span}}_{\cat{P}}(\cat{C})$
  over $\cat{P}$ (so mapping $X$ to
  $\func{S}(X)\colon\cat{P}^{\inrt}_{\func{f}X/}\to\cat{C}$)
  corresponds to
  $\widetilde{\func{S}}
  \colon\cat{X}\times_{\cat{P}}\cat{Ar}_{\inrt}(\cat{P}) \to\cat{C}$
  mapping
  \begin{equation}
    \label{eq:fibration-hom-curry}
    (X,\func{f}X\rightarrowtail Q)\mapsto
    \func{S}(X)(\func{f}X\rightarrowtail Q)\text{.}
  \end{equation}
  In addition, by the description of $\overline{\func{p}}$-cocartesian
  morphisms in $\overline{\cat{Span}}_{\cat{P}}(\cat{C})$ provided
  by~\cite[Corollary 3.2.2.13]{lurie09:_higher}, one sees that an
  $\infty$-functor
  $\func{S}\colon\cat{X}\to\overline{\cat{Span}}_{\cat{P}}(\cat{C})$
  takes an arrow $\xi\colon X\to X^{\prime}$ to a cocartesian arrow in
  $\overline{\cat{Span}}_{\cat{P}}(\cat{C})$ if and only if the
  corresponding $\widetilde{\func{S}}$ takes all morphisms
  $(\xi,\theta)$ where $\theta$ is $\ev_{0}$-cartesian in
  $\cat{Ar}_{\inrt}(\cat{P})$ to equivalences in $\cat{C}$.

  By~\Cref{pros:localis-fibrtn}, $\funcs{X}{C}$ identifies as the full
  sub-$(\infty,1)$-category of $\funcs{X\times_{P}Ar_{\inrt}(P)}{C}$
  on those $\infty$-functors inverting all morphisms in
  $\mathscr{I}_{\cat{X}}$. More precisely, the equivalence
  of~\cref{eq:funcs-span-hom-fibrtn} sits in the sequence
  \begin{equation}
    \label{eq:embed-hom-to-span}
    \funcs{X}{C}
    \simeq\funcs{X\times_{P}Ar_{\inrt}(P)}{C}_{(\mathscr{I}_{\cat{X}})}
    \hookrightarrow\funcs{X\times_{P}Ar_{\inrt}(P)}{C}
    \simeq\funcs{X}{\overline{Span}_{P}(C)}_{/\cat{P}}\text{.}
  \end{equation}
  One then only needs to observe that an $\infty$-functor
  $\widetilde{\mathcal{S}}\in\funcs{X\times_{P}Ar_{\inrt}(P)}{C}$,
  corresponding to
  $\mathcal{S}\in\funcs{X}{\overline{Span}_{P}(C)}_{/\cat{P}}$, is in
  $\funcs{X\times_{P}Ar_{\inrt}(P)}{C}_{(\mathscr{I}_{\cat{X}})}$ if
  and only if for any $\xi\colon X\to X^{\prime}$ in $\cat{X}$ that is
  $\func{f}$-cocartesian and any $\theta$ as
  in~\cref{eq:morph-in-basechgd-arinrt} that is $\ev_{0}$-cartesian
  and $\ev_{0}$-over an inert arrow,
  $\widetilde{\func{S}}(\xi,\theta)$ is an equivalence, which is
  exactly the description given above of $\func{S}$ taking
  $\func{f}$-cocartesian morphisms $\func{f}$-over (since
  $\ev_{0}(\theta)=\func{f}(\xi)$) an inert arrow to
  $\overline{\func{p}}$-cocartesian arrows.
\end{proof}

We can now arrive at our main result.

\begin{thm}
  \label{thm:monads-spans-segal-objs}
  Let $\cat{X}\to\cat{P}$ be a fibrous $\cat{P}$-pattern, with
  $\cat{P}$ sound, and $\cat{C}$ a $\cat{P}$-complete
  $(\infty,1)$-category. There is an equivalence of
  $(\infty,1)$-categories
  \begin{equation}
    \label{eq:main-equiv}
    \seg{\cat{X}}(\cat{C})
    \simeq\cat{Alg}_{\cat{X}}(\cat{Span}_{\cat{P}}(\cat{C}))\text{.}
  \end{equation}

  In particular, taking $\cat{X}=\cat{P}$ to be the terminal fibrous
  $\cat{P}$-pattern, we obtain~\cref{thm:main-result}.
\end{thm}

\begin{proof}
  Let
  $\func{S}\colon\cat{X}\to\overline{\cat{Span}}_{\cat{P}}(\cat{C})$
  be an $\infty$-functor over $\cat{P}$ that preserves cocartesian
  morphisms over inert arrows (so corresponds to
  $\widetilde{\func{S}}\colon\cat{X}\to\cat{C}$). It factors through
  $\cat{Span}_{\cat{P}}(\cat{C})$ if and only if for every
  $X\in\cat{X}$, the $\cat{P}^{\inrt}_{\func{f}X/}$-object
  $\func{S}(X)$ in $\cat{C}$ is Segal.

  By~\cite[Lemma 2.39]{barkan22:_arity_approx_operad}, for any map of
  algebraic patterns $\cat{O}\to\cat{P}$ and any $P\in\cat{P}$, the
  projection $\cat{O}\times_{\cat{P}}\cat{P}_{P/}\to\cat{P}$ is an
  iso-Segal morphism. Applying this to $\cat{O}=\cat{P}^{\inrt}$ and
  $P=\func{f}X$ (for any $X$), we find that the above condition is
  equivalent to $\widetilde{\func{S}}$ being a Segal $\cat{X}$-object.
\end{proof}

\section{Some examples: flavours of generalised multicategories}
\label{sec:some-exampl-flav}

\begin{remark}[Graphs and endomorphisms]
  Since the Segal condition for a pattern $\cat{P}^{\elem}$ with only
  elementary objects and inert morphisms is trivial, the underlying
  $\cat{P}$-graph of the $\cat{P}$-monoidal $(\infty,1)$-category
  $\internalcat{Span}_{\cat{P}}(\cat{C})$ is
  $\internalcat{Span}_{\cat{P}^{\elem}}(\cat{C})$, which is directly
  given by the $\infty$-functor $\funcs{P^{\elem}_{-/}}{C}$. Since
  $\cat{P}^{\elem}_{E/}$, for any elementary $E$, generally has a
  simple form, this will make the underlying $\cat{P}$-graph of
  $\cat{P}$-spans easy to describe.

  Furthermore, since the ``algebraic operations'' in Segal
  $\cat{P}$-objects come from active morphisms, a
  $\cat{P}^{\elem}$-monad carries no algebraic structure and can
  simply be seen as a $\cat{P}$-endomorphism. The statement
  of~\Cref{thm:monads-spans-segal-objs} thus restricts to saying that
  $\cat{P}$-endomorphisms in $\internalcat{Span}_{\cat{P}}(\cat{C})$
  are exactly $\cat{P}$-graphs in $\cat{C}$.
\end{remark}

\subsection{Categories and multiple categories}
\label{sec:multiple-cats-iter-spans}

Take $\cat{P}$ to be the pattern ${\op{\bbDelta}}^{\natural}$,
consisting of the simplicial indexing category $\op{\bbDelta}$ with
its usual inert-active factorisation system (where a map $[n]\to[m]$
in $\bbDelta$ is inert if it is a subinterval inclusion and active if
it is endpoints-preserving), and $[0]$ and $[1]$ as elementary
objects. Its Segal objects are internal categories.

\begin{subremark}
  Direct comparison shows that for any $[n]\in\op{\bbDelta}$, the
  category $({\op{\bbDelta}}^{\natural})^{\inrt}_{[n]/}$ is equivalent
  (in fact isomorphic) to the twisted arrow category of
  $\mathbb{n+1}=[n]$, as has been previously noticed in~\cite[Remark
  5.4]{haugseng18:_iterat} and implicitly used in~\cite[Remark
  2.18]{kositsyn21:_compl}: more precisely, a morphism in
  $\cat{Tw}(\mathbb{n+1})$ represented by a factorising square in
  $\mathbb{n+1}$ below-left
  \begin{equation}
    \label{eq:compar-delta-inrtslice-twisted}
    \begin{tikzcd}
      i \arrow[d,"i\leq j"'] & i^{\prime}
      \arrow[d,"i^{\prime}\leq j^{\prime}"] \arrow[l,"i^{\prime}\leq i"'] \\
      j \arrow[r,"j\leq j^{\prime}"'] & j^{\prime}
    \end{tikzcd}\qquad\qquad
    \begin{tikzcd}
      {[n]} \arrow[r,equals] & {[n]} \\
      {[j-i]}\simeq\{i,\dots,j\} \arrow[r,hook] \arrow[u,hook] &
      {[j^{\prime}-i^{\prime}]}\simeq\{i^{\prime},\dots,j^{\prime}\}
      \arrow[u,hook]
    \end{tikzcd}
  \end{equation}
  corresponds to the morphism in
  $({\op{\bbDelta}}^{\natural})^{\inrt}_{[n]/}$ represented as the
  commutative square (in $\bbDelta$) above-right.
\end{subremark}

Thus for any $(\infty,1)$-category $\cat{C}$ admitting finite fibre
products, $\internalcat{Span}_{{\op{\bbDelta}}^{\natural}}(\cat{C})$
is the double $(\infty,1)$-category of spans in $\cat{C}$ constructed
in~\cite{barwick13:_q} and~\cite{haugseng18:_iterat} (and denoted
$\cat{SPAN}_{1}^{+}(\cat{C})$ there).

Now, we also note that fibrous ${\op{\bbDelta}}^{\natural}$-patterns
are virtual double $\infty$-categories (also referred to as
generalised non-symmetric $\infty$-operads in~\cite{gepner15:_enric})
so that morphisms of fibrous ${\op{\bbDelta}}^{\natural}$-patterns
correspond to ``lax double functors'', and in particular
${\op{\bbDelta}}^{\natural}$-monads recover the usual notion of monad
in a virtual double $(\infty,1)$-category. In conclusion,
\Cref{thm:monads-spans-segal-objs} applied to the pattern
${\op{\bbDelta}}^{\natural}$ recovers the main theorem
of~\cite{haugseng21:_segal}, that monads (or algebras) in spans are
internal categories.

\begin{subexmp}
  \label{exmp:d-uple-monads}
  More generally, using products of algebraic patterns
  (\emph{cf.}~\Cref{exmp:product-pattrn}), one sees that for any
  $d\in\mathbb{N}$, the Segal
  ${\op{\bbDelta}}^{\natural,d}$-$(\infty,1)$-category
  $\internalcat{Span}_{{\op{\bbDelta}}^{\natural,d}}(\cat{C})$ is the
  $(d+1)$-uple $(\infty,1)$-category $\cat{SPAN}_{d}^{+}(\cat{C})$ of
  iterated spans also constructed in~\cite{haugseng18:_iterat}.

  We now explain how lax Segal
  ${\op{\bbDelta}}^{\natural,d}$-fibrations should be seen as virtual
  $(d+1)$-uple $\infty$-categories. When viewing (strong) Segal
  ${\op{\bbDelta}}^{\natural,d}$-fibrations as $(d+1)$-uple
  categories, one should separate the $d$ directions coming from
  $\bbDelta^{d}$, which we dub the \textbf{algebraic} directions, from
  the last one coming from straightening the cocartesian fibration,
  which we will know as the categorical, or \textbf{transversal},
  direction. A lax Segal ${\op{\bbDelta}}^{\natural,d}$-fibration
  $\cat{X}$ is then virtual in all the algebraic directions: it has,
  for all $n\leq d$, algebraic $n$-cells in the usual directions for
  $d$-uple categories, and it has transversal cells from any
  $n$-dimensional grid of $n$-cells to a single $n$-cell. We stress
  that, for the domains of the transversal $n$-cells, we only require
  grids rather than the more general $n$-uple pasting diagrams
  of~\cite{ruit22:_segal}, as the grids are the objects of
  $\bbDelta^{d}$.

  Let us represent the low dimensions; for ease of viewing we shall
  draw the transversal direction vertically, from top to bottom (since
  drawing it transversally would hide the face with the most
  information in the back).

  For $d=1$, the description --- of virtual double $\infty$-categories
  --- is well-known: there are objects and algebraic arrows, and in
  addition there are transversal arrows between objects and
  transversal cells from any pasting diagram (\emph{i.e.} composable
  sequence) of algebraic arrows to one algebraic arrow, drawn as
  $2$-cells in
  \begin{equation}
    \label{eq:virt-dbl-cat-cell}
    \begin{tikzcd}
      \cdot \arrow[r] \arrow[d] & \cdot \arrow[r,""{name=T,below}] &
      \cdot \arrow[r] &
      \cdot \arrow[d] \\
      \cdot \arrow[rrr,""{name=B,above}]
      \arrow[from=T,to=B,Rightarrow,shorten=6] & & & \cdot\text{.}
    \end{tikzcd}
  \end{equation}

  For $d=2$, we similarly have objects, two kinds of algebraic
  $1$-arrows, and algebraic squares or $2$-arrows, and in addition
  transversal $1$-arrows between objects, two kinds of transversal
  $2$-cells, corresponding to the two directions of algebraic arrows,
  and finally transversal cubes or $3$-cells for any grid of
  composable squares, as represented in
  \begin{equation}
    \label{eq:virt-trpl-cat-cell}
    \begin{tikzcd}
      \cdot \arrow[r,""{name=c11T,below,near end}] \arrow[dr]
      \arrow[dd] \arrow[ddrr,phantom,""{name=DiagTop,below}] & \cdot
      \arrow[r,""{name=c12T,below,near end}] \arrow[dr] & \cdot
      \arrow[r,""{name=c13T,below,near end}] \arrow[dr] & \cdot
      \arrow[r,""{name=c14T,below,near end}] \arrow[dr] & \cdot
      \arrow[dr] \\
      & \cdot \arrow[r,""{name=c11B,above,near
        start},""{name=c21T,below,near end}] \arrow[dr]
      \arrow[from=c11T,to=c11B,Rightarrow,shorten=6] & \cdot
      \arrow[r,""{name=c12B,above,near start},""{name=c22T,below,near
        end}] \arrow[dr]
      \arrow[from=c12T,to=c12B,Rightarrow,shorten=6] & \cdot
      \arrow[r,""{name=c13B,above,near start},""{name=c23T,below,near
        end}] \arrow[dr]
      \arrow[from=c13T,to=c13B,Rightarrow,shorten=6] & \cdot
      \arrow[r,""{name=c14B,above,near start},""{name=c24T,below,near
        end}] \arrow[dr]
      \arrow[from=c14T,to=c14B,Rightarrow,shorten=6] & \cdot
      \arrow[dr] \\
      \cdot \arrow[ddrr,""{name=DiagBot,above}]
      \arrow[from=DiagTop,to=DiagBot,Rightarrow,shorten=12] & & \cdot
      \arrow[r,""{name=c21B,above,near start}] \arrow[dd]
      \arrow[rrrr,phantom,""{name=Top,below}]
      \arrow[from=c21T,to=c21B,Rightarrow,shorten=6] & \cdot
      \arrow[r,""{name=c22B,above,near start}]
      \arrow[from=c22T,to=c22B,Rightarrow,shorten=6] & \cdot
      \arrow[r,""{name=c23B,above,near start}]
      \arrow[from=c23T,to=c23B,Rightarrow,shorten=6] & \cdot
      \arrow[r,""{name=c24B,above,near start}]
      \arrow[from=c24T,to=c24B,Rightarrow,shorten=6] & \cdot
      \arrow[dd]
      \\
      & \phantom{\cdot} & & & & \phantom{\cdot} & \\
      & & \cdot \arrow[rrrr,""{name=Bot,above}]
      \arrow[from=Top,to=Bot,Rightarrow,shorten=12] & & & & \cdot
    \end{tikzcd}
  \end{equation}
  where the $3$-cell is not visible but fills the cube.
  
  A ${\op{\bbDelta}}^{\natural,d}$-monad then consists of monads (whose
  structure cells are transversal) in all possible algebraic
  directions and throughout the different dimensions, resembling (a
  less lax version of) the intermonads
  of~\cite[\S{}7.1]{grandis17:_inter}.
\end{subexmp}

\begin{subexmp}
  If one takes instead the pattern ${\op{\bbDelta}}^{\flat}$, which
  has the same underlying category and factorisation system but only
  $[1]$ as elementary object --- whose Segal objects are internal
  categories $X_{\bullet}$ with trivial object $X_{0}$ of objects, so
  internal associative monoids --- then the monoidal
  $(\infty,1)$-category
  $\internalcat{Span}_{{\op{\bbDelta}}^{\flat}}(\cat{C})$, for
  $\cat{C}$ admitting finite products (for this is what
  ${\op{\bbDelta}}^{\flat}$-completeness means) is $\cat{C}$ itself
  seen with its cartesian monoidal structure.

  Generalising to ${\op{\bbDelta}}^{\flat,n}$ (whose Segal objects are
  $n$-iterated associative monoids, so $\mathcal{E}_{n}$-monoids), we
  have that $\internalcat{Span}_{{\op{\bbDelta}}^{\flat,n}}(\cat{C})$
  is $\cat{C}$ seen with its cartesian structure as an
  $\mathcal{E}_{n}$-monoidal structure.  In this case,
  \Cref{thm:monads-spans-segal-objs} simply recovers the fact that
  Segal ${\op{\bbDelta}}^{\flat,n}$-objects in a cartesian
  $(\infty,1)$-category $\cat{C}$ are $n$-uply commutative (meaning
  $\mathcal{E}_{n}$-) algebras in the cartesian monoidal
  $(\infty,1)$-category $\cat{C}^{\times}$ (\emph{i.e.}
  $\mathcal{E}_{n}$-monoids in $\cat{C}$).
\end{subexmp}

\subsection{Commutative monoids}
\label{sec:commutative-monoids}

Take $\cat{P}$ to be the pattern ${\op{\bbGamma}}^{\flat}$ where
$\op{\bbGamma}\simeq\mathfrak{Fin}_{\ast}$ is the opposite of Segal's
category, which is equivalent to the category of pointed finite sets,
with its usual inert-active factorisation system, and
$\langle1\rangle$ as the only elementary object. Its Segal objects are
commutative (or $\mathcal{E}_{\infty}$) monoids. As explained
in~\cite[Example 14.22]{chu21:_homot_segal}, this algebraic pattern is
not saturated; however its global saturation is easily seen from the
fact that ${\op{\bbGamma}}^{\flat,\elem}_{\langle n\rangle/}$ is a set
of $n$ elements.

It also follows that for any ${\op{\bbGamma}}^{\flat}$-complete
(\emph{i.e.} admitting finite products) $(\infty,1)$-category
$\cat{C}$, $\internalcat{Span}_{{\op{\bbGamma}}^{\flat}}(\cat{C})$ is
again $\cat{C}$ itself equipped with its cartesian symmetric monoidal
structure. Since fibrous ${\op{\bbGamma}}^{\flat}$-patterns are
$\infty$-operads in the sense of~\cite{lurie17:_higher_algeb} and
${\op{\bbGamma}}^{\flat}$-monads are commutative algebras, we recover
that Segal ${\op{\bbGamma}}^{\flat}$-objects in $\cat{C}$ are
commutative monoids in $\cat{C}$ (where it is again understood that
the term ``monoid'' refers to an algebra in a cartesian monoidal
$\infty$-category).

\begin{subremark}
  For the product pattern
  $\cat{P}={\op{\bbGamma}}^{\flat}\times{\op{\bbDelta}}^{\natural}$,
  whose Segal objects are internal symmetric monoidal categories, we
  recover as $\internalcat{Span}_{\cat{P}}(\cat{C})$ the double
  $(\infty,1)$-category of spans in $\cat{C}$, endowed with its
  symmetric monoidal structure coming from the cartesian product in
  $\cat{C}$.
\end{subremark}

\begin{subexmp}
  As a further variant, one may consider the algebraic pattern
  ${\op{\bbGamma}}^{\natural}$, which is like
  ${\op{\bbGamma}}^{\flat}$ but also has $\langle0\rangle$ as an
  additional elementary object. Its Segal objects in a
  ${\op{\bbGamma}}^{\natural}$-complete $(\infty,1)$-category
  $\cat{C}$ are commutative monoids in a slice of $\cat{C}$, which it
  is convenient to interpret as families of commutative monoids in
  $\cat{C}$ indexed by an object of $\cat{C}$. In the same spirit,
  fibrous ${\op{\bbGamma}}^{\natural}$-patterns are generalised
  $\infty$-operads of~\cite{lurie17:_higher_algeb}, which are the same
  thing as families of $\infty$-operads.

  For any object $\langle n\rangle$, the category
  ${\op{\bbGamma}}^{\natural,\elem}_{\langle n\rangle/}$ is
  \begin{equation}
    \label{eq:gamma-natrl-elem-slice}
    \begin{tikzcd}
      \rho_{1} & \rho_{2} & \cdots & \rho_{n-1} & \rho_{n}\\
      & & \bigl(\langle n\rangle\xrightarrow{!}\langle0\rangle\bigr)
      \arrow[from=ull] \arrow[from=ul] \arrow[from=ur]
      \arrow[from=urr] & &
    \end{tikzcd}
  \end{equation}
  where $\rho_{i}\colon\langle n\rangle\to\langle1\rangle$ sends $i$
  to $1$ and all the other elements of $\langle n\rangle$ to $0$, from
  which it is seen that the pattern ${\op{\bbGamma}}^{\natural}$ is
  globally saturated. More generally, any inert map
  $\langle n\rangle\to\langle k\rangle$ (with necessarily $k\leq n$)
  determines and is uniquely determined by a $k$-element subset of
  $n=\langle n\rangle\setminus\{0\}$, so that writing $\wp(n)$ for the
  powerset of $n$ (equipped with its natural order), we have
  ${\op{\bbGamma}}^{\natural,\elem}_{\langle n\rangle/}
  \simeq\op{\wp(n)}$. For example, for $n=3$ the poset
  ${\op{\bbGamma}}^{\natural,\elem}_{\langle3\rangle/}$ is
  \begin{equation}
    \label{eq:slice3}
    \begin{tikzcd}
      & & \{1,2,3\} \arrow[dl] \arrow[d] \arrow[dr] & & \\
      & \{1,2\} \arrow[dl] & \{1,3\} \arrow[dll] \arrow[drr] & \{2,3\}
      \arrow[dl,crossing over] \arrow[dr] \\
      \{1\} \arrow[drr] & & \{2\} \arrow[from=ul,crossing over]
      \arrow[d] & &
      \{3\} \arrow[dll] \\
      & & \{0\} & &
    \end{tikzcd}
  \end{equation}
  (containing copies of
  ${\op{\bbGamma}}^{\natural,\elem}_{\langle 2\rangle/}$,
  ${\op{\bbGamma}}^{\natural,\elem}_{\langle 1\rangle/}$, and
  ${\op{\bbGamma}}^{\natural,\elem}_{\langle 0\rangle/}$ on the
  left). We can thus see that
  $\internalcat{Span}_{{\op{\bbGamma}}^{\natural}}(\cat{C})$ is the
  family of slices of $\cat{C}$, each equipped with its monoidal
  structure given by the pullbacks in $\cat{C}$,
  and~\Cref{thm:monads-spans-segal-objs} recovers the description of
  Segal ${\op{\bbGamma}}^{\natural}$-objects given above.
\end{subexmp}

\subsection{Higher categories and iterated spans}
\label{sec:higher-categories-iter-spans}

We now take $\cat{P}$ to be the pattern
${\op{\Theta_{\ell}}}^{\natural}$ of~\Cref{remark:theta-finite-dim},
for some $\ell\in\mathbb{N}\cup\{\omega\}$.

It follows from the description given
in~\cref{eq:theta-elem-slice-over-globe} that
$\internalcat{Span}_{{\op{\Theta}_{\ell}}^{\natural}}(\cat{C})$ is a
cellular $(\infty,1)$-category of $\ell$-times iterated spans: for any
$k\leq\ell$, the $(\infty,1)$-category of $k$-cells has as objects the
spans between the apices of two $(k-1)$-iterated spans, and as
morphisms the morphisms between spans. In other words, it is a
categorical enhancement of the $(\infty,\ell+1)$-category
$\cat{Span}_{\ell}^{+}(\cat{C})$ of $\ell$-iterated spans
from~\cite[Definition 5.16, Remark 5.17]{haugseng18:_iterat}, obtained
by discarding all the extraneous ``algebraic'' directions of the
$(\ell+1)$-uple one as in [\emph{ibid.}] but still retaining the
transversal one.

\begin{subremark}
  At the level of the underlying ${\op{\Theta}}^{\natural}$-graph, the
  fact that our construction of the globular category of iterated
  spans through slices of
  ${\op{\Theta}}^{\natural,\elem}=\op{\mathbb{G}}$ recovers the
  combinatorial one given in~\cite[Definition
  3.2]{batanin98:_monoid_globul_categ_as_natur} was already observed
  in~\cite{street00}.
\end{subremark}

\begin{subexmp}
  For any $k\leq\ell$, we can also define the pattern
  $(\op{\Theta_{\ell}})^{\Sigma^{k}\natural}$ to consist of the same
  structure as ${\op{\Theta}}^{\natural}$ but only the globes
  $\cat{D}_{n}$ with $n\geq k$ as elementaries. For example, if $\ell$
  is finite, taking $k=\ell$ recovers the pattern denoted
  ${\op{\Theta}_{\ell}}^{\flat}$ in~\cite{chu21:_homot_segal}. Segal
  objects for $(\op{\Theta_{\ell}})^{\Sigma^{k}\natural}$ are
  $\mathcal{E}_{k}$-monoidal internal $(\ell-k)$-categories.
\end{subexmp}

As noted in~\cite[Example 9.8. (iv)]{chu21:_homot_segal}, fibrous
${\op{\Theta_{\ell}}}^{\natural}$-patterns are an $\infty$-categorical
version of the $\ell$-globular multicategories or many-sorted
$\ell$-globular operads
of~\cite[p. 273]{leinster04:_higher_operad_higher_categ}
and~\cite[Example 4.11]{cruttwell10}, themselves a many-sorted, or
coloured, version of the $\ell$-globular operads
of~\cite{batanin98:_monoid_globul_categ_as_natur}. They are similar to
the fibrous ${\op{\bbDelta}}^{\natural,\ell}$-patterns described
in~\Cref{exmp:d-uple-monads}, but where the domain of a transversal
$n$-cell is an $n$-categorical pasting diagram instead of an
$n$-dimensional grid (and its codomain is a single $n$-globe rather
than an $n$-cube).

\begin{subwarning}
  Despite their name of ``$\ell$-operads''
  in~\cite{batanin98:_monoid_globul_categ_as_natur}, fibrous
  ${\op{\bbDelta}}^{\natural,\ell}$-patterns should not be thought
  of as a kind of $(\infty,\ell)$-operads, meaning $\infty$-operads
  enriched in $\inflcats{\ell-1}$. Indeed, as seen from the
  description above, they contain more data and structure than
  $(\infty,\ell)$-operads.

  Likewise, the strong Segal
  ${\op{\Theta_{\ell}}}^{\natural}$-fibrations, known as ``monoidal
  $\ell$-globular categories'' in~[\emph{ibid.}], are really
  categorical $(\infty,\ell)$-categories. In particular,
  ${\op{\Theta_{\ell}}}^{\natural}$-monads are very different from any
  kind of usual $\ell$-categorical monads that could be made sense of
  (for example following the philosophy
  of~\cite{haugseng18:_theta_segal} identifying Segal
  ${\op{\Theta_{\ell+1}}}^{\natural}$-objects with reduced categorical
  ${\op{\Theta_{\ell}}}^{\natural}$-objects) in categorical
  $(\infty,\ell)$-categories: the
  ${\op{\Theta_{\ell}}}^{\natural}$-monad structure associates to any
  configuration of (algebraic) $n$-cells a \emph{transversal} cell, so
  is really independent of the $(\infty,\ell)$-categorical structure.

  As such, we will only refer to
  ${\op{\Theta_{\ell}}}^{\natural}$-monads as \textbf{$\ell$-globular
    monads}.
\end{subwarning}

We then obtain by applying~\Cref{thm:monads-spans-segal-objs} that
${\op{\Theta_{\ell}}}^{\natural}$-monads in the categorical
$(\infty,\ell)$-category
$\internalcat{Span}_{{\op{\Theta_{\ell}}}^{\natural}}(\cat{C})$ are
Segal ${\op{\Theta_{\ell}}}^{\natural}$-objects in $\cat{C}$, or in
more evocative language:

\begin{subcorlr}
  There is an equivalence of $(\infty,1)$-categories between
  $\ell$-globular monads in the categorical $(\infty,\ell)$-category
  $\internalcat{Span}_{\ell}^{+}(\cat{C})$ of $\ell$-iterated spans in
  $\cat{C}$, and internal $(\infty,\ell)$-categories in $\cat{C}$.
\end{subcorlr}

For $\ell=\omega$, this recovers a homotopical formulation of the
definition of weak $\omega$-categories given
by~\cite{batanin98:_monoid_globul_categ_as_natur} as well as that
of~\cite{leinster04:_higher_operad_higher_categ} (\emph{cf.}
\cite{cheng04:_higher_dimen_categ} for an explanation of the different
definitions of $\omega$-categories).

\subsection{Multicategories and multispans}
\label{sec:multicategories-multispans}

We finish by considering the algebraic pattern
${\op{\bbOmega}}^{\natural}$
(resp. ${\op{\bbOmega}}_{\text{pl.}}^{\natural}$) whose Segal objects
are internal coloured operads (resp. internal coloured planar
operads). Here, $\bbOmega$ is the dendroidal category, whose objects
are rooted trees (resp. with planar structure), henceforth referred to
as dendrices to avoid confusion with the objects of $\Theta$, and
whose morphisms express the grafting of dendrices --- in contrast with
the morphisms of $\Theta$ which express the pasting of trees. The
algebraic pattern structure is given by having the inert morphisms be
the sub-dendrex inclusions, the active morphisms the
boundary-preserving maps, and the elementary objects be the corollas
$\star_{a}$ (determined by their arities $a\in\mathbb{N}$) and the
nodeless edge $\eta$. As noted in~\cite[Examples
14.21]{chu21:_homot_segal}, the pattern ${\op{\bbOmega}}^{\natural}$
is saturated because any dendrex can be decomposed as a gluing of
corollas along edges, and the same argument shows that it is also
globally saturated.

To understand the dendroidal $(\infty,1)$-category
$\internalcat{Span}_{{\op{\bbOmega}}^{\natural}}(\cat{C})$, let us
first describe its underlying categorical
${\op{\bbOmega}}^{\natural}$-graph
$\internalcat{Span}_{{\op{\bbOmega}}^{\natural,\elem}}(\cat{C})$. At
the level of colours, we just have
$\bbOmega^{\natural,\elem}_{/\eta}=\{\id_{\eta}\}$. At the level of
operations, writing $e_{1},\dots,e_{a}$ the leaves of the corolla
$\star_{a}$ and $r$ its root, we find that
$\bbOmega^{\natural,\elem}_{/\star_{a}}$ is the category
\begin{equation}
  \label{eq:omega-slice-corolla}
  \begin{tikzcd}
    (\eta\xrightarrow{\characmap{e_{1}}}\star_{a}) \arrow[dr] & \cdots
    &
    (\eta\xrightarrow{\characmap{e_{a}}}\star_{a}) \arrow[dl] \\
    & \id_{\star_{a}} & \\
    & (\eta\xrightarrow{\characmap{r}}\star_{a}) \arrow[u] &
  \end{tikzcd}
\end{equation}
of $(a+1)$-ary multicospans.

The structure of category objects in multicategories (coloured
non-symmetric operads) was studied in~\cite[Definition
3.9]{cheng14:_cyclic}. In our case, we get for
$\internalcat{Span}_{{\op{\bbOmega}}^{\natural}}(\cat{C})$ an operadic
composition of multispans by fibre products along the relevant legs,
where each multispan has a distinguished root as seen
in~\cref{eq:omega-slice-corolla}.

\begin{subexmp}
  \label{exmp:cyclic-modular}
  It is also possible to replace ${\op{\bbOmega}}^{\natural}$ by the
  pattern ${\op{\Xi}}^{\natural}$ of~\cite{hackney19:_higher}, whose
  Segal objects are cyclic operads. We obtain for
  $\internalcat{Span}_{{\op{\Xi}}^{\natural}}(\cat{C})$ the same
  structure as above, except that the $(a+1)$-ary spans come without a
  choice of root. Note also that in $\Xi$, the nodeless edge $\eta$ is
  equipped with an involution, which for Segal objects becomes a
  ``duality'' operation on colours. In our case, it acts as the
  identity.

  Going further, we may also use the pattern
  ${\op{\bbUpsilon}}^{\natural}$ of~\cite{hackney20} (denoted
  $\mathbb{U}$ there), whose Segal objects are modular operads. The
  categorical modular $\infty$-operad
  $\internalcat{Span}_{{\op{\bbUpsilon}}^{\natural}}(\cat{C})$ works
  much as $\internalcat{Span}_{{\op{\Xi}}^{\natural}}(\cat{C})$, but
  with additional contraction operations that turn the abstract
  self-duality of objects into an actual self-duality (in the usual
  monoidal, or rather properadic, sense).
\end{subexmp}

The fibrous ${\op{\bbOmega}}^{\natural}$-patterns were identified
in~\cite{berger22:_momen} as the ``tree-hyperoperads'', which are
cumbersome to describe in detail (\emph{cf.}
\cite[\S{}4.1]{getzler98:_modul} or~\cite[Definition
5.45]{markl02:_operad_algeb_topol_physic} for the modular
generalisation, simply called hyperoperads). Nevertheless, we still
obtain from~\Cref{thm:monads-spans-segal-objs} that dendroidal monads
in the categorical $\infty$-operad of multispans in $\cat{C}$ are
internal operads in $\cat{C}$.

\begin{subremark}
  \label{exmp:catl-oprd-multispans}
  The definition of operads, and more general multicategorical
  structures, as monads in multispans is well-known: it dates
  to~\cite{burroni71:_t_catég}, and was independently rediscovered by
  both~\cite{hermida04:_fibrat}
  and~\cite{leinster98:_gener_operad_multic}, and then further
  systematised by~\cite{leinster04:_higher_operad_higher_categ}
  and~\cite{cruttwell10}. From a cartesian monad $\func{T}$ on a
  category $\cat{C}$, one constructs a double category of Kleisli
  $\func{T}$-spans, whose objects are those of $\cat{C}$ and morphisms
  from $C$ to $D$ are spans from $\func{T}C$ to $D$, composition of
  spans using the monad structure.

  For example, taking $\func{T}$ to be the monad
  $\func{F}_{{\op{\bbGamma}}^{\flat}}$ for free monoids, a Kleisli
  $\func{T}$-span is a multispan of arbitrary arity, and monads (in
  the double-categorical sense) are coloured operads. Generally
  speaking, if $\func{T}$ is the monad $\func{F}_{\cat{P}}$ for free
  Segal $\cat{P}$-objects on $\cat{P}$-graphs for some appropriate
  algebraic pattern $\cat{P}$, we expect that monads in Kleisli
  $\func{F}_{\cat{P}}$-spans should be fibrous
  $\cat{P}$-patterns, obtained as the Segal objects for a plus
  construction $\cat{P}^{+}$ of $\cat{P}$ (as in~\cite[Proposition
  3.2.10]{kern23:_monoid_groth_segal}), so as $\cat{P}^{+}$-monads in
  $\cat{P}^{+}$-spans.

  However, Kleisli $\func{F}_{{\op{\bbGamma}}^{\flat}}$-spans and
  ${\op{\bbOmega}}^{\natural}$-spans, while both admitting a natural
  interpretation as multispans, form markedly different structures. On
  the one hand,
  $\internalcat{Span}_{{\op{\bbOmega}}^{\natural}}(\cat{C})$ is a
  categorical $\infty$-operad, whose operadic composition is given
  (leg by leg) by simple pullbacks. On the other hand, Keisli
  $\func{F}_{{\op{\bbGamma}}^{\flat}}$-spans only form a double
  category, but its composition is more complex and makes full use of
  the monad structure on $\func{F}_{{\op{\bbGamma}}^{\flat}}$. For a
  general algebraic pattern $\cat{P}$, the difference will be similar:
  we think of it as moving the structure from the microcosm (on the
  Keisli $\func{F}_{\cat{P}}$-spans side) to the macrocosm (on the
  $\cat{P}^{+}$-spans side).

  It is nonetheless unclear what the precise relation between the two
  constructions is, if there even is one: the Keisli-type construction
  can be abstracted away from a span setting by using general monads
  acting on virtual double categories, but it is unlikely to be
  able to handle non-directed structures such as the cyclic and
  modular $\infty$-operads of~\Cref{exmp:cyclic-modular}.
\end{subremark}

\section{Conclusion: A fibrational perspective}
\label{sec:fibr-persp}

\begin{notation}
  In this~\namecref{sec:fibr-persp}, we will identify
  $(\infty,1)$-categories with internal categories in $\infgrpds$,
  where internal categories are by definition Segal
  ${\op{\bbDelta}}^{\natural}$-objects satisfying Rezk's
  univalence-completeness condition. It will also be convenient to see
  Segal $\cat{P}$-objects in $\infcats$ --- such as, in particular,
  the $\cat{P}$-$(\infty,1)$-categories of $\cat{P}$-spans --- as
  internal categories in $\seg{\cat{P}}(\infgrpds)$.
\end{notation}

Recall that, for any regular cardinal $\kappa$, the
$(\infty,1)$-category $\infgrpds^{(\kappa)}\subset\infgrpds$ is the
base of the universal discrete cocartesian fibration with
$\kappa$-small fibres
$\infgrpds^{(\kappa)}_{\bullet}\to\infgrpds^{(\kappa)}$ in the
$(\infty,2)$-topos $\infcats$, just as
$\cat{Set}^{(\kappa)}\subset\cat{Set}$ is the universal $\kappa$-small
discrete cocartesian fibration in the $(2,2)$-topos $\cat{Cat}$.
In~\cite[Examples 4.7 and 4.8]{weber07:_yoned_struc}, it is explained
that, for an algebraic pattern $\cat{P}^{\elem}$ in which all objects
are elementary and all morphisms inert, the construction
$\internalcat{Span}_{\cat{P}^{\elem}}(-)$ preserves classifying
discrete fibrations, so that the $2$-topos
$\cat{Cat}\bigl(\funcs{\cat{P}^{\elem}}{Set}\bigr)$ has a sufficient
family of classifying discrete cocartesian fibrations given by
$\internalcat{Span}_{\cat{P}^{\elem}}(\cat{Set}_{\bullet}^{(\kappa)})
\to\internalcat{Span}_{\cat{P}^{\elem}}(\cat{Set}^{(\kappa)})$ (where
``sufficient'' means that every discrete cocartesian fibration is
classified by one in the family).

In the $\infty$-categorical setting, the properties of universal (or
``classifying'') fibrations are captured by the notion of univalence,
which we restate from~\cite{gepner16:_unival} (see also~\cite[Theorem
4.4]{rasekh21:_unival_higher_categ_theor}) in the internal setting.

\begin{construct}
  Let $\cat{C}$ be a finitely complete $(\infty,1)$-category. Recall
  that a \textbf{discrete cocartesian fibration} in $\cat{C}$ is an
  internal functor $\func{f}\colon\internalcat{E}\to\internalcat{B}$
  such that
  $(d_{1},\func{f}_{1})\colon\internalcat{E}_{1}\to
  \internalcat{E}_{0}\times_{\internalcat{B}_{0}}\internalcat{B}_{1}$
  is an equivalence.

  Lifting the construction of~\cite[Theorem 2.10]{gepner16:_unival} to
  the cartesian closed $(\infty,2)$-category $\cat{Cat}(\cat{C})$, one
  can construct for any discrete cocartesian fibration
  $\func{f}\colon\internalcat{E}\to\internalcat{B}$ in $\cat{C}$ an
  internal category
  $\internalcat{Eq}_{/\internalcat{B}\times\internalcat{B}}(
  \varpi_{1}^{\ast}\internalcat{E},\varpi_{2}^{\ast}\internalcat{E})$
  over $\mathbb{B}\times\mathbb{B}$ (where
  $\varpi_{1},\varpi_{2}\colon\mathbb{B}\times\mathbb{B}\to\mathbb{B}$
  are the two projections) characterising equivalences between fibres
  of $\func{f}$.
\end{construct}

\begin{definition}[Univalent fibration]
  A discrete cocartesian fibration $\internalcat{E}\to\internalcat{B}$
  internal to $\cat{C}$ is \textbf{univalent} if
  $\internalcat{B}\to\internalcat{Eq}_{/\internalcat{B}\times\internalcat{B}}(
  \varpi_{1}^{\ast}\internalcat{E},\varpi_{2}^{\ast}\internalcat{E})$
  is an equivalence.
\end{definition}

\begin{exmp}
  It is shown in~\cite[Proposition
  5.3.13]{cisinski19:_higher_categ_homot_algeb} that the universal
  discrete cocartesian fibration $\infgrpds_{\bullet}\to\infgrpds$ (in
  the $(\infty,2)$-category $\infcats=\cat{Cat}(\infgrpds)$) is
  univalent.
\end{exmp}

Using this characterisation, one can show (though we omit the proof
here as this result is only used for motivation) that for any sound
algebraic pattern $\cat{P}^{\inrt}$ all of whose morphisms are inert,
the construction
$\internalcat{Span}_{\cat{P}^{\inrt}}(-)\colon\infcats^{(\text{$\cat{P}$-cplt})}
\to\cat{Cat}(\seg{\cat{P}^{\inrt}}(\infgrpds))$ preserves classifying
discrete cocartesian fibrations:

\begin{pros}
  Let $\cat{G}_{\bullet}\to\cat{G}$ be a univalent discrete
  cocartesian fibration. Then
  $\internalcat{Span}_{\cat{P}^{\inrt}}(\cat{G}_{\bullet})
  \to\internalcat{Span}_{\cat{P}^{\inrt}}(\cat{G})$ is a univalent
  discrete cocartesian fibration internally to
  $\seg{\cat{P}^{\inrt}}(\cat{G})$. \qed{}
\end{pros}

For an algebraic pattern with non-trivial active morphisms, the
situation becomes richer and goes beyond the $(\infty,2)$-topos theory
of internal categories in presheaf $(\infty,1)$-topoi. Indeed,
\Cref{thm:monads-spans-segal-objs} shows that, even when restricting
our attention as we are doing here from fibrous patters to Segal
fibrations, the morphisms of interest will be the lax morphisms, the
maps of underlying fibrous patterns. We will thus use a notion of lax
univalence, obtained by replacing strong morphisms by general lax
morphisms of categorical Segal $\cat{P}$-$\infty$-groupoids in the
definition of univalence for fibrations in $\seg{\cat{P}}(\infgrpds)$.

\begin{conj}
  \label{conj:span-preserves-lax-univ-fibr}
  Let $\cat{G}_{\bullet}\to\cat{G}$ be a univalent discrete
  cocartesian fibration. Then
  $\internalcat{Span}_{\cat{P}}(\cat{G}_{\bullet})
  \to\internalcat{Span}_{\cat{P}}(\cat{G})$ is a lax-univalent
  discrete cocartesian fibration internally to
  $\seg{\cat{P}}(\cat{G})$.
\end{conj}

This~\nameCref{conj:span-preserves-lax-univ-fibr} states that
$\internalcat{Span}_{\cat{P}}(\cat{G}_{\bullet})
\to\internalcat{Span}_{\cat{P}}(\cat{G})$ classifies a class of
discrete cocartesian fibrations. It remains to see that \emph{every}
such class is classified by a universal fibration of this form.

\begin{conj}
  \label{conj:span-enough-lax-univ-fibr}
  Suppose
  $\bigl(\cat{G}^{(\kappa)}_{\bullet}\to\cat{G}^{(\kappa)}\bigr)_{\kappa\in
    K}$ is a sufficient family of univalent fibrations for
  $\infcats$. Then the family
  $\bigl(\internalcat{Span}_{\cat{P}}(\cat{G}^{(\kappa)}_{\bullet})
  \to\internalcat{Span}_{\cat{P}}(\cat{G}^{(\kappa)})\bigr)_{\kappa\in
    K}$ provides enough lax-univalent fibrations for
  $\cat{Cat}(\seg{\cat{P}}(\infgrpds))$.
\end{conj}

\begin{corlr}
  \label{corlr:intrnl-fibr-laxmorph-span}
  Let $\cat{P}$ be a sound algebraic pattern and $\cat{X}\to\cat{P}$
  be a Segal $\cat{P}$-fibration, and assume
  that~\Cref{conj:span-preserves-lax-univ-fibr}
  and~\Cref{conj:span-enough-lax-univ-fibr} hold. Then Segal
  $\cat{X}$-objects in $\infgrpds$ are internal discrete cocartesian
  fibrations over the straightening $\mathbb{X}$ of $\cat{X}$.
\end{corlr}

\begin{proof}
  The key point is that, by~\cite[Proposition 3.8]{gepner16:_unival},
  if $\cat{G}^{(\kappa)}_{\bullet}\to\cat{G}^{(\kappa)}$ is univalent
  then $\cat{G}^{(\kappa)}$ is a full sub-$(\infty,1)$-category of
  $\infgrpds$, from which it follows that lax morphisms
  $\cat{X}\to\cat{Span}_{\cat{P}}(\cat{G}^{(\kappa)})$ can be seen as
  lax morphisms $\cat{X}\to\cat{Span}_{\cat{P}}(\infgrpds)$. By the
  two conjectures, discrete cocartesian fibrations over
  $\internalcat{X}$ are the same thing as lax morphisms
  $\cat{X}\to\cat{Span}_{\cat{P}}(\infgrpds)$ (factoring through some
  $\cat{Span}_{\cat{P}}(\cat{G}^{(\kappa)})$). At the same time,
  by~\Cref{thm:monads-spans-segal-objs}, the latter are the same thing
  as Segal $\cat{X}$-objects in $\infgrpds$ (or, to be precise, in
  some $\cat{G}^{(\kappa)}$), which proves the result.
\end{proof}

\begin{exmp}[Double fibrations and Segal fibrations]
  For the algebraic pattern, \Cref{corlr:intrnl-fibr-laxmorph-span}
  says explicitly that discrete cocartesian fibrations of double
  $\infty$-categories correspond biunivocally to lax double
  $\infty$-functors to the double $\infty$-category of spans (of
  $\infty$-groupoids). This is precisely an $\infty$-categorical
  version of the main construction
  of~\cite{lambert21:_discr_doubl_fibrat}. This use of internal
  discrete fibrations is also very similar to
  how~\cite{rasekh21:_cartes_fibrat_compl_segal_spaces} deals with
  fibrations of Segal spaces (see
  also~\cite[\S{}6.1.1]{loubaton23:_theor} for the version for
  fibrations of $\omega$-categories).
\end{exmp}



\vspace{5mm}
\noindent
David Kern\\
KTH Royal Institute of Technology\\
Department of Mathematics\\
SE-100 44 Stockholm (Sweden)\\
dkern@kth.se



\begin{thebibliography}{GHN17}
\bibitem[AF18]{ayala18:_flagg} David Ayala and John Francis. ``Flagged
  higher categories''. In: \emph{Topology and Quantum Theory in
    Interaction}. Vol. 718. Topology and Quantum Theory in
  Interaction. 2018, pp. 137--174. ISBN: 978-1-4704-4243-9. DOI:
  \href{https://doi.org/10.1090/conm/718}{\texttt{10.1090/conm/718}}.
  ar$\chiup$iv:
  \href{https://arxiv.org/abs/1801.08973}{\texttt{1801.08973
      [math.CT]}}.
\bibitem[Ara10]{ara10:_sur_groth} Dimitri Ara. ``Sur les
  $\infty$-groupoïdes de Grothendieck et une variante
  $\infty$-catégorique''. PhD thesis. Université Paris Diderot (Paris
  7), 2010. URL:
  \href{https://www.imj-prg.fr/theses/pdf/dimitri_ara.pdf}
  {\texttt{https://www.imj-prg.fr/theses/pdf/dimitri\_ara.pdf}}.
\bibitem[Baa19]{baas19} Nils A. Baas. ``On the mathematics of higher
  structures''. In: \emph{International Journal of General Systems}
  48.6 (2019), pp. 603--624. DOI:
  \href{https://doi.org/10.1080/03081079.2019.1615906}
  {\texttt{10.1080/03081079.2019.1615906}}.
  ar$\chiup$iv:
  \href{https://arxiv.org/abs/1805.11944}{\texttt{1805.11944
      [math.GM]}}.
\bibitem[Bar13]{barwick13:_q} Clark Barwick. ``On the Q construction
  for exact quasicategories'' (2013). ar$\chiup$iv:
  \href{https://arxiv.org/abs/1301.4725}{\texttt{1301.4725
      [math.KT]}}.
\bibitem[Bar22]{barkan22:_arity_approx_operad} Shaul Barkan. ``Arity
  Approximation of $\infty$-Operads'' (2022). ar$\chiup$iv:
  \href{https://arxiv.org/abs/2207.07200}{\texttt{2207.07200
      [math.AT]}}.
\bibitem[Bat98]{batanin98:_monoid_globul_categ_as_natur} Michael
  Batanin. ``Monoidal Globular Categories As a Natural Environment for
  the Theory of Weak $n$-Categories''. In: \emph{Advances in
    Mathematics} 136 (1998), pp. 39--103. DOI:
  \href{https://doi.org/10.1006/aima.1998.1724}
  {\texttt{10.1006/aima.1998.1724}}.
\bibitem[Bén67]{bénabou67:_introd} Jean Bénabou. ``Introduction to
  bicategories''. In: \emph{Reports of the Midwest Category
    Seminar}. Vol. 47. Lecture Notes in Mathematics. Springer Berlin
  Heidelberg, 1967, pp. 1--77. DOI:
  \href{https://doi.org/10.1007/BFb0074299}{\texttt{10.1007/BFb0074299}}.
\bibitem[Ber02]{berger02:_cellul_nerve_higher_categ} Clemens
  Berger. ``A Cellular Nerve for Higher Categories''. In: Advances in
  Mathematics 169 (2002), pp. 118--175. DOI:
  \href{https://doi.org/10.1006/aima.2001.2056}
  {\texttt{10.1006/aima.2001.2056}}.
\bibitem[Ber22]{berger22:_momen} Clemens Berger. ``Moment categories
  and operads''. In: \emph{Theory and Applications of Categories}
  38.39 (2022). DOI:
  \href{https://doi.org/10.70930/tac/6g3jhwj0}
  {\texttt{10.70930/tac/6g3jhwj0}}. ar$\chiup$iv:
  \href{https://arxiv.org/abs/2102.00634}{\texttt{2102.00634
      [math.CT]}}.
\bibitem[BHS22]{barkan22:_envel_algeb_patter} Shaul Barkan, Rune
  Haugseng and Jan Steinebrunner. ``Envelopes for Algebraic
  Patterns''. (2022) ar$\chiup$iv:
  \href{https://arxiv.org/abs/2208.07183}{\texttt{2208.07183
      [math.CT]}}.
\bibitem[Bur71]{burroni71:_t_catég} Albert Burroni. ``$T$-catégories
  (Catégories dans un triple)''. In: \emph{Cahiers de topologie et
    géométrie différentielle} 12.3 (1971), pp. 215--321. URL:
  \href{http://www.numdam.org/item/CTGDC_1971__12_3_215_0/}
  {\texttt{http://www.numdam.org/item/CTGDC\_1971\_\_12\_3\_215\_0/}}.
\bibitem[BS$\geq$25]{barkan25:_sound_algeb_patter} Shaul Barkan and
  Jan Steinebrunner. ``Skew-Segal morphisms and
  soundness''. (Forthcoming)
\bibitem[CGR14]{cheng14:_cyclic} Eugenia Cheng, Nick Gurski and Emily
  Riehl. ``Cyclic multicategories, multivariable adjunctions and
  mates''. In: \emph{Journal of K-Theory} 13.2 (2014),
  pp. 337--396. DOI:
  \href{https://doi.org/10.1017/is013012007jkt250}
  {\texttt{10.1017/is013012007jkt250}}.
  ar$\chiup$iv:
  \href{https://arxiv.org/abs/1208.4520}{\texttt{1208.4520
      [math.CT]}}.
\bibitem[CH21]{chu21:_homot_segal} Hongyi Chu and Rune
  Haugseng. ``Homotopy-coherent algebra via Segal conditions''. In:
  \emph{Advances in Mathematics} 385 (2021), p. 107733. ISSN:
  0001-8708. DOI:
  \href{https://doi.org/10.1016/j.aim.2021.107733}
  {\texttt{10.1016/j.aim.2021.107733}}.
  ar$\chiup$iv:
  \href{https://arxiv.org/abs/1907.03977}{\texttt{1907.03977
      [math.AT]}}.
\bibitem[Cis19]{cisinski19:_higher_categ_homot_algeb} Denis-Charles
  Cisinski. \emph{Higher Categories and Homotopical
    Algebra}. Cambridge Studies in Advanced Mathematics. Cambridge
  University Press, 2019. ISBN: 9781108588737. DOI:
  \href{https://doi.org/10.1017/9781108588737}
  {\texttt{10.1017/9781108588737}}. 
  URL:
  \href{https://cisinski.app.uniregensburg.de/CatLR.pdf}
  {\texttt{https://cisinski.app.uniregensburg.de/CatLR.pdf}}.
\bibitem[CL04]{cheng04:_higher_dimen_categ} Eugenia Cheng and Aaron
  Lauda. \emph{Higher-Dimensional Categories: an illustrated guide
    book}. 2004. URL:
  \href{https://eugeniacheng.com/wp-content/uploads/2017/02/cheng-lauda-guidebook.pdf}
  {\texttt{https://eugeniacheng.com/wp-content/uploads\\
      /2017/02/cheng-lauda-guidebook.pdf}}.
\bibitem[CS10]{cruttwell10} Geoffrey S.H. Cruttwell and Michael
  A. Shulman. ``A unified framework for generalized
  multicategories''. In: \emph{Theory and Applications of Categories}
  24.21 (2010), pp. 580--655. DOI:
  \href{https://doi.org/10.70930/tac/mxocppsl}
  {\texttt{10.70930/tac/mxocppsl}}. ar$\chiup$iv:
  \href{https://arxiv.org/abs/0907.2460}{\texttt{0907.2460
      [math.CT]}}.
\bibitem[GH15]{gepner15:_enric} David Gepner and Rune
  Haugseng. ``Enriched $\infty$-categories via non-symmetric
  $\infty$-operads''. In: \emph{Advances in Mathematics} 279 (2015),
  pp. 575--716. ISSN: 0001-8708. DOI:
  \href{https://doi.org/10.1016/j.aim.2015.02.007}
  {\texttt{10.1016/j.aim.2015.02.007}}. ar$\chiup$iv:
  \href{https://arxiv.org/abs/1312.3178}{\texttt{1312.3178
      [math.AT]}}.
\bibitem[GHN17]{gepner17:_lax_colim_free_fibrat_categ} David Gepner,
  Rune Haugseng and Thomas Nikolaus. ``Lax Colimits and Free
  Fibrations in $\infty$-Categories''. In: \emph{Documenta
    Mathematica} 22 (2017), pp. 1225--1266. DOI:
  \href{https://doi.org/10.4171/DM/593}{\texttt{10.4171/DM/593}}. ar$\chiup$iv:
  \href{https://arxiv.org/abs/1501.02161}{\texttt{1501.02161
      [math.CT]}}.
\bibitem[GK16]{gepner16:_unival} David Gepner and Joachim
  Kock. ``Univalence in locally cartesian closed
  $\infty$-categories''. In: \emph{Forum Mathematicum} 29 (2016),
  pp. 617--652. DOI:
  \href{https://doi.org/10.1515/forum-2015-0228}
  {\texttt{10.1515/forum-2015-0228}}.
  ar$\chiup$iv:
  \href{https://arxiv.org/abs/1208.1749}{\texttt{1208.1749
      [math.CT]}}.
\bibitem[GK98]{getzler98:_modul} Ezra Getzler and Mikhail
  M. Kapranov. ``Modular operads''. In: \emph{Compositio Mathematica}
  110.1 (1998), pp. 65--125. ISSN: 1570-5846. DOI:
  \href{https://doi.org/10.1023/A:1000245600345}
  {\texttt{10.1023/A:1000245600345}}. ar$\chiup$iv:
  \href{https://arxiv.org/abs/dg-ga/9408003}{\texttt{dg-ga/9408003
      [dg-ga]}}.
\bibitem[GP17]{grandis17:_inter} Marco Grandis and Robert
  Paré. ``Intercategories: A framework for three-dimensional category
  theory''. In: \emph{Journal of Pure and Applied Algebra} 221 (2017),
  pp. 999--1054. DOI:
  \href{https://doi.org/doi.org/10.1016/j.jpaa.2016.08.002}
  {\texttt{doi.org/10.1016/j.jpaa.2016.08.002}}. ar$\chiup$iv:
  \href{https://arxiv.org/abs/1412.0212}{\texttt{1412.0212
      [math.CT]}}.
\bibitem[Hau18a]{haugseng18:_iterat} Rune Haugseng. ``Iterated spans
  and classical topological field theories''. In: \emph{Mathematische
    Zeitschrift} 289 (2018), pp. 1427--1488. DOI:
  \href{https://doi.org/10.1007/s00209-017-2005-x}
  {\texttt{10.1007/s00209-017-2005-x}}. ar$\chiup$iv:
  \href{https://arxiv.org/abs/1409.0837}{\texttt{1409.0837
      [math.AT]}}.
\bibitem[Hau18b]{haugseng18:_theta_segal} Rune Haugseng. ``On the
  equivalence between $\Theta_{n}$-spaces and iterated Segal
  spaces''. In: \emph{Proceedings of the American Mathematical
    Society} 146.4 (2018), pp. 1401--1415. DOI:
  \href{https://doi.org/10.1090/proc/13695}
  {\texttt{10.1090/proc/13695}}. ar$\chiup$iv:
  \href{https://arxiv.org/abs/1604.08480}{\texttt{1604.08480
      [math.AT]}}.
\bibitem[Hau21]{haugseng21:_segal} Rune Haugseng. ``Segal spaces,
  spans, and semicategories''. In: \emph{Proceedings of the American
    Mathematical Society} 149.3 (2021), pp. 961--975. DOI:
  \href{https://doi.org/10.1090/proc/15197}
  {\texttt{10.1090/proc/15197}}. ar$\chiup$iv:
  \href{https://arxiv.org/abs/1901.08264}{\texttt{1901.08264
      [math.AT]}}.
\bibitem[Her04]{hermida04:_fibrat} Claudio Hermida. ``Fibrations for
  abstract multicategories''. In: \emph{Galois Theory, Hopf
    Algebras, and Semiabelian Categories}. Vol. 43. Fields Institute
  Communications, 2004, pp. 281--293. URL:
  \href{http://sqig.math.ist.utl.pt/pub/HermidaC/fibmul.pdf}
  {\texttt{http://sqig.math.ist.utl.pt/pub/HermidaC/fibmul.pdf}}.
\bibitem[HRY19]{hackney19:_higher} Philip Hackney, Marcy Robertson and
  Donald Yau. ``Higher cyclic operads''. In: \emph{Algebraic \&
    Geometric Topology} 19 (2019), pp. 863--940. DOI:
  \href{https://doi.org/10.2140/agt.2019.19.863}
  {\texttt{10.2140/agt.2019.19.863}}. ar$\chiup$iv:
  \href{https://arxiv.org/abs/1611.02591}{\texttt{1611.02591
      [math.AT]}}.
\bibitem[HRY20]{hackney20} Philip Hackney, Marcy Robertson and Donald
  Yau. ``A graphical category for higher modular operads''. In:
  \emph{Advances in Mathematics} 365 (2020), p. 107044. ISSN:
  0001-8708. DOI: \href{https://doi.org/10.1016/j.aim.2020.107044}
  {\texttt{10.1016/j.aim.2020.107044}}. ar$\chiup$iv:
  \href{https://arxiv.org/abs/1906.01143}{\texttt{1906.01143
      [math.AT]}}.
\bibitem[Joy97]{joyal97:_disks_theta} André Joyal. ``Disks, duality
  and $\Theta$-categories''. 1997. URL:
  \href{https://ncatlab.org/nlab/files/JoyalThetaCategories.pdf}
  {\texttt{https://ncatlab.org/nlab/files/JoyalThetaCategories.pdf}.}
\bibitem[Ker23]{kern23:_monoid_groth_segal} David Kern. ``Monoidal
  envelopes and Grothendieck construction for dendroidal Segal
  objects'' (2023). ar$\chiup$iv:
  \href{https://arxiv.org/abs/2301.10751}{\texttt{2301.10751
      [math.CT]}}.
\bibitem[Kos21]{kositsyn21:_compl} Roman Kositsyn. ``Completeness for
  monads and theories'' (2021). ar$\chiup$iv:
  \href{https://arxiv.org/abs/2104.00367}{\texttt{2104.00367
      [math.CT]}}.
\bibitem[Lam21]{lambert21:_discr_doubl_fibrat} Michael
  J. Lambert. ``Discrete Double Fibrations''. In: \emph{Theory and
    Applications of Categories} 37.22 (2021), pp. 671--708. DOI: \href{https://doi.org/10.70930/tac/zvxpxw4p}
  {\texttt{10.70930/tac/zvxpxw4p}}. ar$\chiup$iv:
  \href{https://arxiv.org/abs/2101.06734}{\texttt{2101.06734
      [math.CT]}}.
\bibitem[Lei04]{leinster04:_higher_operad_higher_categ} Tom
  Leinster. \emph{Higher Operads, Higher Categories}. London
  Mathematical Society Lecture Note Series. Cambridge University
  Press, 2004. DOI: \href{https://doi.org/10.1017/CBO9780511525896}
  {\texttt{10.1017/CBO9780511525896}}. ar$\chiup$iv:
  \href{https://arxiv.org/abs/math/0305049}{\texttt{math/0305049
      [math.CT]}}.
\bibitem[Lei98]{leinster98:_gener_operad_multic} Tom
  Leinster. ``General Operads and Multicategories'' (1998). ar$\chiup$iv:
  \href{https://arxiv.org/abs/math/9810053}{\texttt{math/9810053
      [math.CT]}}.
\bibitem[Lou23]{loubaton23:_theor} Félix Loubaton. ``Theory and
    models of $(\infty,\omega)$-categories''. PhD thesis. Laboratoire
    J. A. Dieudonné, 2023. ar$\chiup$iv:
    \href{https://arxiv.org/abs/2307.11931}{\texttt{2307.11931
        [math.CT]}}.
    URL: \href{https://theses.hal.science/tel-04308414}
    {\texttt{https://theses.hal.science/tel-04308414}}.
\bibitem[Lur09]{lurie09:_higher} Jacob Lurie. \emph{Higher Topos
    Theory}. Vol. 170. Annals of Mathematics Studies. Princeton
  University Press, Princeton, NJ, 2009, pp. xviii+925. ISBN:
  978-0-691-14049-0. DOI: \href{https://doi.org/10.1515/9781400830558}
  {\texttt{10.1515/9781400830558}}.
\bibitem[Lur17]{lurie17:_higher_algeb} Jacob Lurie. \emph{Higher
    Algebra}. 2017. URL:
  \href{http://math.ias.edu/~lurie/papers/HA.pdf}
  {\texttt{http://math.ias.edu/~lurie/papers/HA.pdf}}.
\bibitem[MSS02]{markl02:_operad_algeb_topol_physic} Martin Markl,
  Steve Shnider and Jim Stasheff. \emph{Operads in Algebra, Topology
    and Physics}. Vol. 96. Mathematical Surveys and
  Monographs. 2002. DOI:
  \href{https://doi.org/10.1090/surv/096}{\texttt{10.1090/surv/096}}.
\bibitem[Ras21a]{rasekh21:_cartes_fibrat_compl_segal_spaces} Nima
  Rasekh. ``Cartesian Fibrations of Complete Segal Spaces''
  (2021). ar$\chiup$iv:
  \href{https://arxiv.org/abs/2102.05190}{\texttt{2102.05190
      [math.CT]}}.
\bibitem[Ras21b]{rasekh21:_unival_higher_categ_theor} Nima
  Rasekh. ``Univalence in Higher Category Theory'' (2021). ar$\chiup$iv:
  \href{https://arxiv.org/abs/2103.12762}{\texttt{2103.12762
      [math.CT]}}.
\bibitem[Rez10]{rezk10:_cart} Charles Rezk. ``A cartesian presentation
  of weak $n$–categories''. In: \emph{Geometry \& Topology} 14.1
  (2010), pp. 521--571. DOI:
  \href{https://doi.org/10.2140/gt.2010.14.521}
  {\texttt{10.2140/gt.2010.14.521}}. ar$\chiup$iv:
  \href{https://arxiv.org/abs/0901.3602}{\texttt{0901.3602
      [math.CT]}}.
\bibitem[Rui22]{ruit22:_segal} Jaco Ruit. ``A pasting theorem for
  iterated Segal spaces'' (2022). ar$\chiup$iv:
  \href{https://arxiv.org/abs/2210.04549}{\texttt{2210.04549
    [math.CT]}}.
\bibitem[RV22]{riehl22:_elemen} Emily Riehl and Dominic
  Verity. \emph{Elements of $\infty$-category
    theory}. Vol. 194. Cambridge Studies in Advanced
  Mathematics. Cambridge University Press, 2022. ISBN:
  9781108936880. DOI: \href{https://doi.org/10.1017/9781108936880}
  {\texttt{10.1017/9781108936880}}. URL:
  \href{https://elementsbook.github.io/elements.pdf}
  {\texttt{https://elementsbook.github.io/elements.pdf}}.
\bibitem[Ste25]{steinebr25:_global_satur} Jan Steinebrunner, personal
  communication, 2025.
\bibitem[Str00]{street00} Ross Street. ``The petit topos of globular
  sets''. In: \emph{Journal of Pure and Applied Algebra} 154 (2000),
  pp. 299--315. DOI:
  \href{https://doi.org/10.1016/S0022-4049(99)00183-8}
  {\texttt{10.1016/S0022-4049(99)00183-8}}.
\bibitem[Web07]{weber07:_yoned_struc} Mark Weber. ``Yoneda Structures
  from 2-toposes''. In: \emph{Applied Categorical Structures} 15
  (2007), pp. 259--323. DOI:
  \href{https://doi.org/10.1007/s10485-007-9079-2}
  {\texttt{10.1007/s10485-007-9079-2}}.
  ar$\chiup$iv:
  \href{https://arxiv.org/abs/math/0606393}{\texttt{math/0606393
      [math.CT]}}.
\end{thebibliography}
\end{document}